\documentclass[11pt,a4paper,reqno]{amsproc}
\usepackage[british]{babel}

\usepackage{a4wide}
\usepackage{setspace}
\usepackage{amssymb}
\usepackage{amsmath}
\usepackage{amsthm}
\usepackage{enumitem}
\usepackage{mathrsfs}
\usepackage[colorlinks,citecolor=blue,urlcolor=black,linkcolor=red,linktocpage]{hyperref}
\usepackage{mathrsfs}
\usepackage{mathtools}

\newtheorem{thm}{Theorem}[section]
\newtheorem{cor}[thm]{Corollary}
\newtheorem{lem}[thm]{Lemma}

\newtheorem{prop}[thm]{Proposition}

\theoremstyle{definition}
\newtheorem{problem}[thm]{Problem}

\newtheorem{definition}[thm]{Definition}

\newtheorem{remark}[thm]{Remark}

\renewcommand{\epsilon}{\varepsilon}
\renewcommand{\phi}{\varphi}
\newcommand{\defeq}{\mathrel{\mathop:}=}                         
\newcommand{\eqdef}{\mathrel{\mathopen={\mathclose:}}}
\renewcommand{\Re}{\operatorname{Re}}
\renewcommand{\Im}{\operatorname{Im}}          

\DeclareMathOperator{\spt}{spt}

\DeclareMathOperator{\convc}{\overline{conv}}

\begin{document}

\onehalfspace

\setlist{noitemsep}

\author{Friedrich Martin Schneider}
\address{F.M.S., Institute of Algebra, TU Dresden, 01062 Dresden, Germany }
\email{martin.schneider@tu-dresden.de}

\author{Andreas Thom}
\address{A.T., Institute of Geometry, TU Dresden, 01062 Dresden, Germany }
\email{andreas.thom@tu-dresden.de}

\title[Random walks on topological groups]{The Liouville property and random walks on\\ topological groups}
\date{\today}

\begin{abstract} 
  We study harmonic functions and Poisson boundaries for Borel probability measures on general (i.e., not necessarily locally compact) topological groups, and we prove that a second-countable topological group is amenable if and only if it admits a fully supported, regular Borel probability measure with trivial Poisson boundary. This generalizes work of Kaimanovich--Vershik and Rosenblatt, confirms a general topological version of Furstenberg's conjecture, and entails a characterization of the amenability of isometry groups in terms of the Liouville property for induced actions. Moreover, our result has non-trivial consequences concerning Liouville actions of discrete groups on countable sets.
\end{abstract}

\maketitle


\tableofcontents

\section{Introduction}

The study of random walks on countable discrete (resp., second-countable locally \mbox{compact}) groups via their corresponding \emph{Poisson boundary} was initiated in a series of papers by Furstenberg~\cite{furstenberg63a,furstenberg63b,furstenberg71,furstenberg73} and has since evolved into a major theme in geometric group theory. Furstenberg's original construction of the Poisson boundary for random walks on discrete groups, relying upon the martingale convergence theorem, was soon generalized to locally compact groups~\cite{azencott} and complemented by a number of non-probabilistic approaches, ranging from functional-analytic to \mbox{dynamical}~\cite{zimmer,paterson,dokken,DokkenEllis,willis}. As revealed by the multitude of rather different perspectives on Poisson boundaries, there is no essential obstruction to defining and studying such boundaries for general (i.e., not necessarily locally compact) topological groups~\cite{prunaru}. On the other hand, recent years' growing interest in topological dynamics and ergodic theory of \emph{large} (often referred to as \emph{infinite-dimensional}) Polish groups, such as arising, for instance, in Ramsey theory and operator algebra, suggests to actually pursue this path. The present article aims to initiate and advance this development.

A definite milestone in the classical theory of Poisson boundaries of random walks on groups is marked by Kaimanovich--Vershik's proof~\cite{VershikKaimanovich,KaimanovichVershik} of Furstenberg's conjecture~\cite{furstenberg73}: \emph{a countable group G is amenable if and only if it possesses a probability measure whose support is all of G and whose Poisson boundary is trivial}. It was Rosenblatt~\cite{rosenblatt} who established Furstenberg's conjecture for random walks on $\sigma$-compact locally compact groups. The central objective of the present paper is a general topological version of Furstenberg's conjecture (Corollary~\ref{corollary:main}), valid for arbitrary second-countable groups, which follows from a new characterization of amenability in terms of asymptotic invariance of the sequence of convolutional powers of a suitable fully supported, regular Borel probability measure (Theorem~\ref{theorem:main}). The proof of the latter combines ideas of Kaimanovich and Vershik~\cite{KaimanovichVershik} with some recent work by the present authors~\cite{SchneiderThom}. Furthermore, our result reveals a correspondence between the amenability of topological groups of isometries of metric spaces and the Liouville property for their induced actions (Theorem~\ref{theorem:juschenko}).

Perhaps surprisingly, the topological version of Furstenberg's conjecture (Corollary~\ref{corollary:main}) does have interesting and non-trivial consequences even in the context of actions of countable (discrete) groups on countable sets. Answering a recent question by Juschenko~\cite{juschenko}, we show that, for every $n \in \mathbb{N}$, the action of Thompson's group $F$ on the set $\mathbb{Z}\!\left[ \tfrac{1}{2} \right]$ of dyadic rationals is \emph{$n$-Liouville}~\cite{juschenko}, that is, the induced action of $F$ on the collection of $n$-element subsets of $\mathbb{Z}\!\left[ \tfrac{1}{2} \right]$ admits a Liouville measure. In the light of our Corollary~\ref{corollary:juschenko.2}, this becomes an immediate consequence of Pestov's work~\cite{pestov} establishing the (extreme) amenability of the topological group $\mathrm{Aut}\!\left(\mathbb{Z}\!\left[ \tfrac{1}{2} \right],{\leq}\right)$, which contains $F$ as a dense subgroup.

This article is organized as follows. In Section~\ref{section:ueb} we recall some necessary background concerning the amenability of general topological groups, including its connection with the so-called UEB topology. The subsequent Section~\ref{section:convolution} is dedicated to providing some technical prerequisites on convolution semigroups over topological groups. In Section~\ref{section:poisson} we study harmonic functions and Poisson boundaries for Borel probability measures on topological groups and prove the above-mentioned generalization of Furstenberg's conjecture, the consequences of which for isometry groups form the subject of Section~\ref{section:liouville}. The final Section~\ref{section:thompson} contains applications of our results in the context of discrete group actions and Thompson's group $F$.

\section{UEB topology and amenability}\label{section:ueb}

The purpose of this preliminary section is to fix some notation and to recall some basic concepts concerning function spaces on topological groups. The focus is on providing some necessary background regarding the UEB topology on the corresponding dual vector spaces, including its connection with the amenability of topological groups. For a considerably more comprehensive exposition (capturing the general setting of arbitrary uniform spaces), the reader is referred to Pachl's monograph~\cite{PachlBook}.

Before getting to topological groups, let us briefly clarify some general notation. Given a set $S$, we denote by $\ell^{\infty}(S)$ the commutative unital $C^{\ast}$-algebra of all bounded complex-valued functions on $S$, equipped with the pointwise operations and the supremum norm \begin{displaymath}
	\Vert f \Vert_{\infty} \, \defeq \, \sup \{ \vert f(s) \vert \mid s \in S \} \qquad \left( f \in \ell^{\infty}(S) \right) .
\end{displaymath} Consider any Hausdorff topological space $T$ and let $\mu$ be a Borel probability measure on $T$. The \emph{support} of $\mu$ is defined to be the closed subset of $T$ given by \begin{displaymath}
	\spt (\mu) \, \defeq \, \{ x \in T \mid \forall U \subseteq T \text{ open}\colon \, x \in U \Longrightarrow \mu (U) > 0 \} .
\end{displaymath} Let us call $\mu$ \emph{finitely supported} if $\spt (\mu)$ is finite, and \emph{fully supported} if $\spt (\mu) = T$. As usual, $\mu$ will be called \emph{regular} if $\mu (B) = \sup \{ \mu(K) \mid K\subseteq B, \, K \text{ compact} \}$ for every Borel set $B \subseteq T$. The set of all Borel probability measures on $T$ will be denoted by $\mathrm{Prob}(T)$ and the subset of regular Borel probability measures on $T$ will be denoted by $\mathrm{Prob}_{\mathrm{reg}}(T)$.

Throughout the present article, a \emph{topological group} is always understood to be Hausdorff. Let $G$ be a topological group. A function $f \colon G \to \mathbb{C}$ is said to be \emph{right-uniformly continuous} if for every $\epsilon > 0$ there exists a neighborhood $U$ of the neutral element in $G$ such that \begin{displaymath}
	\forall x,y \in G \colon \qquad xy^{-1} \in U \ \Longrightarrow \ \vert f(x) - f(y) \vert \, \leq \, \epsilon .
\end{displaymath} The set $\mathrm{RUCB}(G)$ of all right-uniformly continuous, bounded complex-valued functions on~$G$, equipped with the pointwise operations and the supremum norm, constitutes a commutative unital $C^{\ast}$-algebra. A set $H \subseteq \mathrm{RUCB}(G)$ is called \emph{UEB} (short for \emph{uniformly equicontinuous, bounded}) if $H$ is $\Vert \cdot \Vert_{\infty}$-bounded and \emph{right-uniformly equicontinuous}, that is, for every $\epsilon > 0$ there exists a neighborhood $U$ of the neutral element in $G$ such that \begin{displaymath}
	\forall f \in H \ \forall x,y \in G \colon \qquad xy^{-1} \in U \ \Longrightarrow \ \vert f(x) - f(y) \vert \, \leq \, \epsilon .
\end{displaymath}
The set $\mathrm{RUEB}(G)$ of all UEB subsets of $\mathrm{RUCB}(G)$ is a convex vector bornology on $\mathrm{RUCB}(G)$. The \emph{UEB topology} on the dual Banach space $\mathrm{RUCB}(G)^{\ast}$ is defined as the topology of uniform convergence on the members of $\mathrm{RUEB}(G)$. Of course, this is a locally convex linear topology on the vector space $\mathrm{RUCB}(G)^{\ast}$ containing the weak-${}^{\ast}$ topology, that is, the initial topology generated by the functions $\mathrm{RUCB}(G)^{\ast} \to \mathbb{C}, \, \mu  \mapsto \mu (f)$ where $f \in \mathrm{RUCB}(G)$. More detailed information on the structure of $\mathrm{RUCB}(G)^{\ast}$ and the UEB topology can be found in~\cite{PachlBook}. For most of this paper, we will be interested in aspects concerning the convex subset \begin{displaymath}
	\mathrm{M}(G) \, \defeq \, \{ \mu \in \mathrm{RUCB}(G)^{\ast} \mid \mu \text{ positive}, \, \mu (\mathbf{1}) = 1 \} ,
\end{displaymath} i.e., the set of \emph{states} on the $C^{\ast}$-algebra $\mathrm{RUCB}(G)$, which forms a compact Hausdorff space with respect to the weak-${}^{\ast}$ topology. The set $\mathrm{S}(G)$ of all ${}^{\ast}$-homomorphisms from $\mathrm{RUCB}(G)$ to $\mathbb{C}$ forms a closed subspace of $\mathrm{M}(G)$, which is called the \emph{Samuel compactification} of $G$. Now, for any $g \in G$, let $\lambda_{g} \colon G \to G, \, x \mapsto gx$ and $\rho_{g} \colon G \to G, \, x \mapsto xg$. Note that the group $G$ acts by linear isometries on the Banach space $\mathrm{RUCB}(G)^{\ast}$ via \begin{displaymath}
	(g \mu)(f) \, \defeq \, \mu (f \circ \lambda_{g}) \qquad \left( g \in G, \, \mu \in \mathrm{RUCB}(G)^{\ast}, \, f \in \mathrm{RUCB}(G) \right) .
\end{displaymath} It is straightforward to check that $\mathrm{M}(G)$ forms a $G$-invariant subset of $\mathrm{RUCB}(G)^{\ast}$ and that the restricted action of $G$ on $\mathrm{M}(G)$ is affine and continuous with respect to the weak-${}^{\ast}$ topology. Moreover, $\mathrm{S}(G)$ constitutes a $G$-invariant subspace of $\mathrm{M}(G)$. We recall that $G$ is said to be \emph{amenable} (resp., \emph{extremely amenable}) if $\mathrm{M}(G)$ (resp., $\mathrm{S}(G)$) admits a $G$-fixed point. It is well known that $G$ is amenable (resp., extremely amenable) if and only if every continuous action of $G$ on a non-void compact Hausdorff space admits a $G$-invariant regular Borel probability measure (resp., a $G$-fixed point). For a more comprehensive account on (extreme) amenability of general topological groups, we refer to~\cite{PestovBook}.

We recall a recent characterization of amenability in terms of asymptotically invariant finitely supported probability measures.

\begin{thm}[\cite{SchneiderThom}, Theorem~3.2]\label{theorem:topological.day} A topological group $G$ is amenable if and only if, for every $\epsilon > 0$, every $H \in \mathrm{RUEB}(G)$ and every finite subset $E \subseteq G$, there exists a finitely supported, regular Borel probability measure $\mu$ on $G$ such that \begin{displaymath}
	\forall g \in E \ \forall f \in H \colon \quad \left\lvert \int f \circ \lambda_{g} \, d\mu - \int f \, d\mu \right\rvert \, \leq \, \epsilon .
\end{displaymath} \end{thm}

\begin{cor}\label{corollary:topological.day} A topological group $G$ is amenable if and only if, for every $\epsilon > 0$, every $H \in \mathrm{RUEB}(G)$ and every finite subset $E \subseteq G$, there exists a finitely supported, regular Borel probability measure $\mu$ on $G$ such that $E \subseteq \spt (\mu)$ and \begin{displaymath}
	\forall g \in E \ \forall f \in H \colon \quad \left\lvert \int f \circ \lambda_{g} \, d\mu - \int f \, d\mu \right\rvert \, \leq \, \epsilon .
\end{displaymath} \end{cor}

\begin{proof} In view of Theorem~\ref{theorem:topological.day}, we only need to prove that the former implies the latter. To this end, let $\epsilon \in (0,1]$ and $H \in \mathrm{RUEB}(G)$, and let $E \subseteq G$ be finite. Without loss of generality, we may and will assume that $E \ne \emptyset$. By Theorem~\ref{theorem:topological.day}, since $G$ is amenable, there exists a finitely supported, regular Borel probability measure $\mu$ on $G$ with $\left\lvert \int f \circ \lambda_{g} \, d\mu - \int f \, d\mu \right\rvert \leq \tfrac{\epsilon}{2}$ for all $f \in H$ and $g \in G$. For $\alpha \defeq \tfrac{\epsilon}{1+4\sup_{f \in H} \Vert f \Vert_{\infty}} \in (0,1]$, we consider the finitely supported, regular Borel probability measure $\nu$ on $G$ given by \begin{displaymath}
	\nu (B) \, \defeq \, (1-\alpha ) \mu (B) + \alpha \tfrac{\vert B \cap E \vert}{\vert E \vert} \qquad \left(\text{$B \subseteq G$ Borel}\right) .
\end{displaymath} Evidently, $E \subseteq \spt (\nu)$. Moreover, \begin{align*}
	\left\lvert \int f \circ \lambda_{g} \, d\nu - \int f \, d\nu \right\rvert \, &\leq \, (1-\alpha) \left\lvert \int f \circ \lambda_{g} \, d\mu - \int f \, d\mu \right\rvert + \tfrac{\alpha}{\vert E \vert} \sum\nolimits_{x \in E} \left\lvert f(gx) - f(x) \right\rvert \\
	&\leq \, (1-\alpha ) \tfrac{\epsilon}{2} + \alpha 2 \Vert f \Vert_{\infty} \, \leq \, \tfrac{\epsilon}{2} + \tfrac{\epsilon}{2} \, = \, \epsilon
\end{align*} for all $f \in H$ and $g \in G$, as desired. \end{proof}

Theorem~\ref{theorem:topological.day} suggests the following Definition~\ref{definition:convergence.to.invariance}.

\begin{definition}\label{definition:convergence.to.invariance} Let $G$ be a topological group. A net $(\mu_{i})_{i \in I}$ in $\mathrm{M}(G)$ is said to \emph{UEB-converge to invariance (over $G$)} if \begin{displaymath}
	\forall g \in G \ \forall H \in \mathrm{RUEB}(G) \colon \quad \sup\nolimits_{f \in H} \lvert \mu_{i}(f \circ \lambda_{g}) - \mu_{i}(f) \rvert \, \longrightarrow \, 0 \quad (i \longrightarrow I) .
\end{displaymath} \end{definition}

\begin{remark}[see~\cite{SchneiderThom}, Proof of Theorem~3.2]\label{remark:convergence.to.invariance} Let $G$ be a topological group. If $(\mu_{i})_{i \in I}$ is a net in $\mathrm{M}(G)$ UEB-converging to invariance over $G$, then any weak-${}^{\ast}$ accumulation point of the net $(\mu_{i})_{i \in I}$ in $\mathrm{M}(G)$ is $G$-invariant. \end{remark}

For later use, we also note the subsequent observation.

\begin{lem}\label{lemma:dense.convergence.to.invariance} Let $S$ be any dense subset of a topological group $G$. A net $(\mu_{i})_{i \in I}$ in $\mathrm{M}(G)$ UEB-converges to invariance over $G$ if and only if \begin{displaymath}
		\forall g \in S \ \forall H \in \mathrm{RUEB}(G) \colon \quad \sup\nolimits_{f \in H} \lvert \mu_{i}(f \circ \lambda_{g}) - \mu_{i}(f) \rvert \, \longrightarrow \, 0 \quad (i \longrightarrow I) .
\end{displaymath} \end{lem}

\begin{proof} Evidently, the former implies the latter. In order to prove the converse implication, let $(\mu_{i})_{i \in I}$ be a net in $\mathrm{M}(G)$ such that \begin{displaymath}
	\forall g \in S \ \forall H \in \mathrm{RUEB}(G) \colon \quad \sup\nolimits_{f \in H} \lvert \mu_{i}(f \circ \lambda_{g}) - \mu_{i}(f) \rvert \, \longrightarrow \, 0 \quad (i \longrightarrow I) .
\end{displaymath} Consider any $g \in G$ and $H \in \mathrm{RUEB}(G)$. We wish to show that $\sup\nolimits_{f \in H} \lvert \mu_{i}(f \circ \lambda_{g}) - \mu_{i}(f) \rvert \longrightarrow 0$ as $i \to I$. To this end, let $\epsilon > 0$. Since $H \in \mathrm{RUEB}(G)$, there exists a neighborhood $U$ of the neutral element in $G$ such that $\Vert f - (f \circ \lambda_{u}) \Vert_{\infty} \leq \tfrac{\epsilon}{2}$ for all $f \in H$ and $u \in U$. As $S$ is dense in $G$, there exists $s \in S$ with $s \in Ug$. Moreover, by assumption, we find $i _{0} \in I$ such that \begin{displaymath}
	\forall i \in I , \, i \geq i_{0} \colon \quad \sup\nolimits_{f \in H} \lvert \mu_{i}(f \circ \lambda_{s}) - \mu_{i}(f) \rvert \, \leq \, \tfrac{\epsilon}{2} . 
\end{displaymath} For every $i \in I$ with $i \geq i_{0}$, we conclude that \begin{align*}
	\lvert \mu_{i}(f \circ \lambda_{g}) - \mu_{i}(f) \rvert \, &\leq \, \lvert \mu_{i}(f \circ \lambda_{g}) - \mu_{i}(f \circ \lambda_{s}) \rvert + \lvert \mu_{i}(f \circ \lambda_{s}) - \mu_{i}(f) \rvert \\
	&\leq \, \lVert (f \circ \lambda_{g}) - (f \circ \lambda_{s}) \rVert_{\infty} + \tfrac{\epsilon}{2} \, = \, \left\lVert f - \left(f \circ \lambda_{sg^{-1}}\right) \right\rVert_{\infty} + \tfrac{\epsilon}{2} \, \leq \, \epsilon 
\end{align*} for all $f \in H$, i.e., $\sup_{f \in H} \lvert \mu_{i}(f \circ \lambda_{g}) - \mu_{i}(f) \rvert \leq \epsilon$ as desired. \end{proof}

Due to well-known work of Birkhoff~\cite{birkhoff} and Kakutani~\cite{Kakutani36}, a topological group $G$ is first-countable if and only if $G$ is metrizable, in which case $G$ moreover admits a metric $d$ both generating the topology of $G$ and being \emph{right-invariant}, i.e., satisfying $d(xg,yg) = d(x,y)$ for all $g,x,y \in G$. Furthermore, if $G$ is any metrizable topological group, then the Birkhoff--Kakutani theorem gives rise to a corresponding description of the UEB topology on norm-bounded subsets of $\mathrm{RUCB}(G)^{\ast}$ (Lemma~\ref{lemma:UEB.metric} below). To be precise, let us clarify some notation. Let $X$ be a metric space. Given any $\ell \in \mathbb{R}_{\geq 0}$, let us consider the set \begin{displaymath}
	\left. \mathrm{Lip}_{\ell} (X) \, \defeq \, \left\{ f \in \mathbb{R}^{X} \, \right\vert \forall x,y \in X \colon \, \vert f(x) - f(y) \vert \leq \ell d_{X}(x,y) \right\}
\end{displaymath} of all \emph{$\ell$-Lipschitz} real-valued functions on $X$, and moreover let $\mathrm{Lip}_{\ell}^{\infty}(X) \defeq \mathrm{Lip}_{\ell}(X) \cap \ell^{\infty}(X)$ and $\mathrm{Lip}_{\ell}^{r}(X) \defeq \mathrm{Lip}_{\ell}(X) \cap [-r,r]^{X}$ for $r \in \mathbb{R}_{\geq 0}$. Define $\mathrm{Lip}^{\infty}(X) \defeq \bigcup_{\ell \in  \mathbb{R}_{\geq 0}} \mathrm{Lip}_{\ell}^{\infty}(X)$.

\begin{lem}\label{lemma:UEB.metric} Let $G$ be a topological group. If $d$ is a right-invariant metric on $G$ generating the topology of $G$, then the UEB topology and the topology generated by the norm  \begin{displaymath}
	\mathrm{p}_{d} \colon \, \mathrm{RUCB}(G)^{\ast} \longrightarrow \, \mathbb{R}, \quad f \, \longmapsto \, \sup \left\{ \vert \mu (f) \vert \left\vert \, f \in \mathrm{Lip}_{1}^{1}(G,d) \right\} \right.
\end{displaymath} coincide on every $\Vert \cdot \Vert$-bounded subset of $\mathrm{RUCB}(G)^{\ast}$. \end{lem}

\begin{proof} Evidently, $\mathrm{p}_{d}$ constitutes a semi-norm on $\mathrm{RUCB}(G)^{\ast}$. Since $d$ is right-invariant and generates the topology of $G$, it follows that $\mathrm{Lip}_{1}^{1}(G,d)$ generates a $\Vert \cdot \Vert_{\infty}$-dense linear subspace of $\mathrm{RUCB}(G)$ (see~\cite[Lemma~5.20(2)]{PachlBook}), whence $\mathrm{p}_{d}$ is in fact a norm on~$\mathrm{RUCB}(G)^{\ast}$. Moreover, as $d$ is right-invariant and continuous, $\mathrm{Lip}_{1}^{1}(G,d)$ is easily seen to be a member of $\mathrm{RUEB}(G)$, thus the topology $\tau_{d}$ generated by $\mathrm{p}_{d}$ is contained in the UEB topology $\tau_{\mathrm{UEB}}$. It remains to prove that, for every $\Vert \cdot \Vert$-bounded $B \subseteq \mathrm{RUCB}(G)^{\ast}$, the restriction of $\tau_{\mathrm{UEB}}$ to $B$ is contained in the restriction of $\tau_{d}$ to $B$. So, consider any $\Vert \cdot \Vert$-bounded subset $B \subseteq \mathrm{RUCB}(G)^{\ast}$. Let $\mu \in B$, $H \in \mathrm{RUEB}(G)$ and $\epsilon > 0$. Then $K \defeq \{ \Re(f) \mid f \in H \} \cup \{ \Im (f) \mid f \in H \}$ belongs to $\mathrm{RUEB}(G)$, too. Again due to $d$ being a right-invariant metric generating the topology~of~$G$, there exists a real $r > 0$ such that \begin{displaymath}
	\forall f \in K \ \exists g \in r \mathrm{Lip}_{1}^{1}(G,d) \colon \quad \Vert f - g \Vert_{\infty} \, \leq \, \tfrac{\epsilon}{1 + 8\sup_{\mu \in B} \Vert \mu \Vert} 
\end{displaymath} (see~\cite[Lemma~5.20(1)]{PachlBook}). Consequently, if $\nu \in B$, then \begin{align*}
	\sup\nolimits_{f \in H} \vert (\mu - \nu) (f) \vert \, &\leq \, 2 \sup\nolimits_{f \in K} \vert (\mu - \nu) (f) \vert \\
	& \leq \, 2\sup\nolimits_{g \in r \mathrm{Lip}_{1}^{1}(G,d)} \vert (\mu - \nu) (g) \vert + \tfrac{\epsilon}{2} \, = \, 2 r\mathrm{p}_{d}(\mu - \nu) + \tfrac{\epsilon}{2} \, .
\end{align*} Hence, $\!\left\{ \nu \in B \left\vert \, \mathrm{p}_{d}(\mu - \nu) \leq \tfrac{\epsilon}{4r} \right\} \subseteq \{ \nu \in B \mid \forall f \in H\colon \, \vert (\mu - \nu) (f) \vert \leq \epsilon \} \right. \!$. This shows that the restriction of $\tau_{\mathrm{UEB}}$ to $B$ is contained in the corresponding restriction of the topology $\tau_{d}$. \end{proof}

\section{Convolution semigroups}\label{section:convolution}

Studying the amenability of general topological groups, we take advantage of the intrinsic right topological semigroup structure present on the continuous duals of the corresponding spaces of uniformly continuous bounded functions. A \emph{right topological semigroup} is a semigroup $S$ together with a topology on $S$ such that $S \to S, \, s \mapsto st$ is continuous for every $t \in S$, and the \emph{topological centre} of $S$ is then defined to be the subset \begin{displaymath}
	\Lambda (S) \, \defeq \, \{ t \in S \mid S \to S, \, s \mapsto ts \ \text{continuous} \} ,
\end{displaymath} which is easily seen to form a subsemigroup of $S$. Before proceeding to semigroups associated to topological groups, let us note the following general fact.

\begin{lem}\label{lemma:semigroup.closure} If $S$ is a right topological semigroup and $T$ is a subsemigroup of $\Lambda (S)$, then $\overline{T}$ is a subsemigroup of $S$. \end{lem}

\begin{proof} Since $T$ is a subsemigroup of $\Lambda (S)$, it follows that $t\overline{T} \subseteq \overline{tT} \subseteq \overline{T}$ for every $t \in T$, which means that $T\overline{T} \subseteq \overline{T}$. Hence, if $t \in \overline{T}$, then $\overline{T}t \subseteq \overline{Tt} \subseteq \overline{\overline{T}} = \overline{T}$, as desired. \end{proof}

We now turn to the natural semigroup structure on $\mathrm{RUCB}(G)^{\ast}$ for an arbitrary topological group $G$. The details of the construction are recorded in the subsequent lemma (recollected from~\cite[Section~2.2]{BerglundJunghennMilnes} and \cite[Section~6.1]{PestovBook}), whose straightforward proof we omit.

\begin{lem}[cf.~Section~2.2 in~\cite{BerglundJunghennMilnes}]\label{lemma:convolution} Let $G$ be a topological group. The following hold. \begin{itemize}
	\item[$(1)$] For any $\mu \in \mathrm{RUCB}(G)^{\ast}$ and $f \in \mathrm{RUCB}(G)$, the function \begin{displaymath}
		\qquad \qquad \Phi_{\mu}f \colon \, G \, \longrightarrow \, \mathbb{C} , \quad g \, \longmapsto \, \mu (f \circ \lambda_{g})
	\end{displaymath} is right-uniformly continuous and bounded with $\Vert \Phi_{\mu}f \Vert_{\infty} \leq \Vert \mu \Vert \Vert f \Vert_{\infty}$.
	\item[$(2)$] For all $\mu \in \mathrm{RUCB}(G)^{\ast}$, $f \in \mathrm{RUCB}(G)$ and $g \in G$, \begin{displaymath}
		\qquad \qquad \Phi_{\mu}(f \circ \lambda_{g}) \, = \, (\Phi_{\mu}f) \circ \lambda_{g} , \qquad \qquad \Phi_{g\mu}f \, = \, (\Phi_{\mu}f) \circ \rho_{g} .
	\end{displaymath}
	\item[$(3)$] Let $\mu \in \mathrm{RUCB}(G)^{\ast}$. Then, $\Phi_{\mu} \colon \mathrm{RUCB}(G) \to \mathrm{RUCB}(G)$ is a bounded linear operator with $\Vert \Phi_{\mu} \Vert = \Vert \mu \Vert$. If $\mu$ is positive (resp., unital, a ${}^{\ast}$-homomorphism), then so is $\Phi_{\mu}$.
	\item[$(4)$] Let $\mu, \nu \in \mathrm{RUCB}(G)^{\ast}$. Then \begin{displaymath}
		\qquad \qquad \mu\nu \colon \, \mathrm{RUCB}(G) \, \longrightarrow \, \mathbb{C}, \quad f \, \longmapsto \, \mu (\Phi_{\nu}f)
	\end{displaymath} belongs to $\mathrm{RUCB}(G)^{\ast}$, $\Vert \mu \nu \Vert \leq \Vert \mu \Vert \Vert \nu \Vert$ and $\Phi_{\mu \nu} = \Phi_{\mu} \circ \Phi_{\nu}$. If both $\mu$ and $\nu$ are positive (resp., unital, ${}^{\ast}$-homomorphisms), then so is $\mu\nu$. In particular, $\mu\nu \in \mathrm{M}(G)$ if $\mu, \nu \in \mathrm{M}(G)$, and $\mu\nu \in \mathrm{S}(G)$ if $\mu, \nu \in \mathrm{S}(G)$.
	\item[$(5)$] For all $\mu, \nu, \xi \in \mathrm{RUCB}(G)^{\ast}$, we have $\mu (\nu \xi) = (\mu \nu) \xi$.
\end{itemize} \end{lem}

Let $G$ be a topological group. It follows that $\mathrm{RUCB}(G)^{\ast}$, together with the multiplication defined in Lemma~\ref{lemma:convolution}(4), constitutes a unital Banach algebra. For each $g \in G$, let us consider the state $\delta_{g} \colon \mathrm{RUCB}(G) \to \mathbb{C}, \, f \mapsto f(g)$. We observe that the multiplication on~$\mathrm{RUCB}(G)^{\ast}$ is compatible with the action of $G$ upon $\mathrm{RUCB}(G)^{\ast}$ introduced in Section~\ref{section:ueb}, in the sense that $g\mu = \delta_{g}\mu$ for all $\mu \in \mathrm{RUCB}(G)^{\ast}$ and $g \in G$. Furthermore, endowed with this multiplication and the weak-${}^{\ast}$ topology, $\mathrm{RUCB}(G)^{\ast}$ is easily seen to form a right topological semigroup, of which $\mathrm{M}(G)$ and $\mathrm{S}(G)$ are subsemigroups by Lemma~\ref{lemma:convolution}(4).

Let us briefly examine the topological centre of the right topological semigroup isolated above. Consider any topological group $G$. The \emph{convolution} of two Borel probability measures $\mu$ and $\nu$ on $G$ is the Borel probability measure $\mu\nu$ on $G$ defined by \begin{displaymath}
	(\mu \nu)(B) \, \defeq \, (\mu \otimes \nu) (\{ (x,y) \in G \mid xy \in B \}) \qquad \left(\text{$B\subseteq G$ Borel}\right) .
\end{displaymath} We note that $\spt (\mu \nu) \subseteq (\spt \mu) (\spt \nu)$ for any two Borel probability measures $\mu$ and $\nu$ on $G$. It is well known and straightforward to check that $\mathrm{Prob}(G)$ together with this convolution forms a semigroup and that $\mathrm{Prob}_{\mathrm{reg}}(G)$ is a subsemigroup thereof. The map $I_{G} \colon \mathrm{Prob}(G) \to \mathrm{M}(G)$ given by \begin{displaymath}
	I_{G}(\mu)(f) \, \defeq \, \int f \, d\mu \qquad \left(\mu \in \mathrm{Prob}(G), \, f \in \mathrm{RUCB}(G)\right) 
\end{displaymath} is a semigroup homomorphism, in particular $I_{G}(\mathrm{Prob}_{\mathrm{reg}}(G))$ and $I_{G}(\mathrm{Prob}(G))$ constitute subsemigroups of $\mathrm{M}(G)$. Since $\mathrm{Prob}_{\mathrm{reg}}(G)$ contains all Dirac measures on $G$, the latter moreover entails that both $I_{G}(\mathrm{Prob}_{\mathrm{reg}}(G))$ and $I_{G}(\mathrm{Prob}(G))$ are $G$-invariant subsets of $\mathrm{M}(G)$.

\begin{remark}\label{remark:measures.vs.functionals} Let $G$ be a topological group. Due to~\cite[Theorem~5.3]{PachlBook}, the restriction of $I_{G}$ to $\mathrm{Prob}_{\mathrm{reg}}(G)$ is injective. Furthermore, if $G$ is metrizable (equivalently, first-countable), then the map $I_{G}$ itself is injective, since in this case every closed subset of $G$ is the zero set of some element of $\mathrm{RUCB}(G)$. In any case, starting from Section~\ref{section:poisson}, we will not distinguish in notation between a Borel probability measure on $G$ and its image under $I_{G}$, and we will adopt terminology introduced so far for elements of $\mathrm{M}(G)$ (such as in Definition~\ref{definition:convergence.to.invariance}) for members of $\mathrm{Prob}(G)$ accordingly. \end{remark}

With the notation introduced above at hand, we recollect two results about topological centres from literature. The first one is a generalization of a result about locally compact groups by Wong~\cite[Lemma~3.1]{wong}.

\begin{prop}[\cite{FerriNeufang}, Proposition~4.2]\label{proposition:topological.centre} For any topological group $G$, \begin{displaymath}
	I_{G}(\mathrm{Prob}_{\mathrm{reg}}(G)) \, \subseteq \, \Lambda(\mathrm{M}(G)) .
\end{displaymath} \end{prop}

The second is a consequence of~\cite[Theorem~7.25]{PachlBook} and~\cite[Theorem~9.41(1)]{PachlBook}.

\begin{prop}[\cite{PachlBook}, Theorems~7.25, 9.41(1)]\label{proposition:topological.centre.pachl} For a separable topological group $G$, \begin{displaymath}
	I_{G}(\mathrm{Prob}(G)) \, \subseteq \, \Lambda(\mathrm{M}(G)) .
\end{displaymath} \end{prop}

If $G$ is a locally compact group, then the inclusion noted in Proposition~\ref{proposition:topological.centre} above is indeed an equality, i.e., $I_{G}(\mathrm{Prob}_{\mathrm{reg}}(G)) = \Lambda(\mathrm{M}(G))$. This follows from the work of Neufang~\cite{neufang} sharpening a result by Lau~\cite{Lau86}. Natural analogues of this result for different classes of topological groups are due to Ferri and Neufang~\cite{FerriNeufang}, and Pachl~\cite{pachl}. For more details, we refer to~\cite[Section~9.4]{PachlBook}.

For the remainder of this section, we turn to a class of function sets on topological groups, classically referred to as introverted sets~\cite{day,GranirerLau}.

\begin{definition} Let $G$ be a topological group and let $H \subseteq \mathrm{RUCB}(G)$. Then $H$ will be called \emph{right-translation closed} if $f \circ \rho_{g} \in H$ for all $f \in H$ and $g  \in G$, and $H$ will be called \emph{introverted} if $\Phi_{\mu}f \in H$ whenever $f \in H$ and $\mu  \in \mathrm{M}(G)$. \end{definition}

When dealing with introverted sets, we will make use of the following observations.

\begin{lem}\label{lemma:introverted} Let $G$ be a topological group and $f \in \mathrm{RUCB}(G)$. The linear map \begin{displaymath}
	\Psi_{f} \colon \, \mathrm{RUCB}(G)^{\ast} \longrightarrow \, \mathrm{RUCB}(G), \quad \mu \, \longmapsto \, \Phi_{\mu}f
\end{displaymath} is continuous with respect to the weak-${}^{\ast}$ topology on $\mathrm{RUCB}(G)^{\ast}$ and the topology of pointwise convergence on $\mathrm{RUCB}(G)$. Furthermore, \begin{itemize}
	\item[$(1)$] $\Psi_{f}(\mathrm{M}(G)) = \convc \{ f \circ \rho_{g} \mid g \in G \}$, and
	\item[$(2)$] $\Psi_{f}(B) \in \mathrm{RUEB}(G)$ for every norm-bounded $B \subseteq \mathrm{RUCB}(G)^{\ast}$.
\end{itemize} \end{lem}

\begin{proof} For every $h \in \mathrm{RUCB}(G)$, let $\eta_{h} \colon \mathrm{RUCB}(G)^{\ast} \longrightarrow \mathbb{C}, \, \mu \mapsto \mu (h)$. If $x \in G$, then \begin{displaymath}
	\left(\delta_{x} \circ \Psi_{f}\right)\! (\mu) = \delta_{x}(\Phi_{\mu}f) = (\Phi_{\mu}f)(x) = \mu (f \circ \lambda_{x}) = \eta_{f \circ \lambda_{x}}(\mu)
\end{displaymath} for all $\mu \in \mathrm{RUCB}(G)^{\ast}$, i.e., $P_{x} \circ \Psi_{f} = Q_{f \circ \lambda_{x}}$. This readily entails the continuity stated above. Since $\Psi_{f}(\delta_{x}) = \Phi_{\delta_{x}}f = f \circ \rho_{x}$ for all $x \in G$ and the convex hull of $\{ \delta_{x} \mid x \in G \}$ is dense in~the compact space $\mathrm{M}(G)$ (see~\cite[Corollary~P.6]{PachlBook}), continuity and linearity of $\Psi_{f}$ imply that \begin{displaymath}
	\Psi_{f}(\mathrm{M}(G)) \, = \, \convc \{ T_{f}(\delta_{x}) \mid x \in G \} \, = \, \convc \{ f \circ \rho_{x} \mid x \in G \} ,
\end{displaymath} which proves~(1). In order to verify~(2), let $B$ be a norm-bounded subset of $\mathrm{RUCB}(G)^{\ast}$ and consider $s \defeq \sup_{\mu \in B} \Vert \mu \Vert < \infty$. Let $\epsilon > 0$. As $f$ is right-uniformly continuous, we find a neighborhood $U$ of the neutral element in $G$ such that $\Vert f - (f \circ \lambda_{g}) \Vert_{\infty} \leq \tfrac{\epsilon}{s+1}$ for all $g \in U$. We claim that \begin{displaymath}
	\forall x,y \in G , \, xy^{-1} \in U \colon \quad \sup\nolimits_{\mu \in B} \vert (\Phi_{\mu}f)(x) - (\Phi_{\mu}f)(y) \vert \leq \epsilon .
\end{displaymath} Indeed, if $\mu \in B$ and $x,y \in G$ with $xy^{-1} \in U$, then \begin{displaymath}
	\vert (\Phi_{\mu}f)(x) - (\Phi_{\mu}f)(y) \vert \, = \, \vert \mu((f \circ \lambda_{x})-(f \circ \lambda_{y})) \vert \, \leq \, s \left\Vert f - \left(f \circ \lambda_{xy^{-1}}\right) \right\Vert_{\infty} \, \leq \, \epsilon \, .
\end{displaymath} This shows that $\Psi_{f}(B) = \{ \Phi_{\mu}f \mid \mu \in B \}$ is right-uniformly equicontinuous. Of course, $\Psi_{f}(B)$ is $\Vert \cdot \Vert_{\infty}$-bounded by Lemma~\ref{lemma:convolution}(4). Hence, $\Psi_{f}(B) \in \mathrm{RUEB}(G)$. \end{proof}

\begin{cor}\label{corollary:introverted} Let $G$ be a topological group and let $H \subseteq \mathrm{RUCB}(G)$. If $H$ is introverted, then $H$ is right-translation invariant. Conversely, if $H$ right-translation closed, convex, and closed w.r.t.~the topology of pointwise convergence, then $H$ is introverted. \end{cor}

\begin{proof} This is an immediate consequence of Lemma~\ref{lemma:introverted}(1). \end{proof}

\begin{cor}\label{corollary:UEB.metric} Let $G$ be a topological group and let $d$ be a right-invariant metric on $G$ generating the topology of $G$. Then $\mathrm{p}_{d}(\mu \nu) \leq \mathrm{p}_{d}(\mu)$ for all $\mu \in \mathrm{RUCB}(G)^{\ast}$ and $\nu \in \mathrm{M}(G)$. \end{cor}

\begin{proof} Evidently, $\mathrm{Lip}_{1}^{1}(G,d)$ is right-translation closed, convex, and closed with respect to the topology of pointwise convergence, hence introverted by Corollary~\ref{corollary:introverted}. Consequently, if $\mu \in \mathrm{RUCB}(G)^{\ast}$ and $\nu \in \mathrm{M}(G)$, then \begin{align*}
	\mathrm{p}_{d}(\mu \nu) \, &= \, \sup \left\{ \vert (\mu \nu) (f) \vert \left\vert \, f \in \mathrm{Lip}_{1}^{1}(G,d) \right\} \right. \, = \, \sup \left\{ \vert \mu (\Phi_{\nu}f) \vert \left\vert \, f \in \mathrm{Lip}_{1}^{1}(G,d) \right\} \right. \\
	&\leq \, \sup \left\{ \vert \mu (f) \vert \left\vert \, f \in \mathrm{Lip}_{1}^{1}(G,d) \right\} \right. \, = \, \mathrm{p}_{d}(\mu) . \qedhere
\end{align*} \end{proof}

The next lemma will allow us to swap quantifiers when testing amenability.

\begin{lem}\label{lemma:amenability} Let $G$ be a topological group, let $\Sigma$ be a closed subsemigroup of $\mathrm{M}(G)$, and let~$H \subseteq \mathrm{RUCB}(G)$. Suppose that $\Phi_{\mu}(H) \subseteq H$ for all $\mu \in \Sigma$. The following are equivalent. \begin{enumerate}
		\item[$(1)$] $\exists \mu \in \Sigma \ \forall f \in H \ \forall g \in G \colon \ \mu (f \circ \lambda_{g}) = \mu (f)$.
		\item[$(2)$] $\forall f \in H \ \exists \mu \in \Sigma \ \forall g \in G \colon \ \mu (f \circ \lambda_{g}) = \mu (f)$.
\end{enumerate} \end{lem}

\begin{proof} Of course, the former implies the latter. Since $\Sigma$ is a compact Hausdorff space, in order to prove the converse implication it suffices to show that \begin{displaymath}
	\forall F \subseteq H \text{ finite} \ \exists \mu \in \Sigma \ \forall g \in G \colon \quad \mu (f \circ \lambda_{g}) \, = \, \mu (f) .
\end{displaymath} Consider any finite subset $F = \{ f_{1},\ldots,f_{n} \} \subseteq H$. By assumption, there is $\mu_{n} \in \Sigma$ such that $\mu_{n}(f_{n} \circ \lambda_{g}) = \mu_{n}(f_{n})$ for all $g \in G$, i.e., the function $\Phi_{\mu_{n}}f_{n}$ is constant. As $\Phi_{\mu}(H) \subseteq H$ for all $\mu \in \Sigma$, we may continue by recursion: for each $i \in \{ 1,\ldots,n-1 \}$, there exists $\mu_{i} \in \Sigma$ such that $\mu_{i}((\Phi_{\mu_{i+1}\cdots\mu_{n}}f_{i}) \circ \lambda_{g}) = \mu_{i}(\Phi_{\mu_{i+1}\cdots\mu_{n}}f_{i})$ for every $g \in G$, that is, $\Phi_{\mu_{i}}(\Phi_{\mu_{i+1}\cdots\mu_{n}}f_{i}) = \Phi_{\mu_{i}\cdots \mu_{n}}f_{i}$ is constant. Finally, let us define $\mu \defeq \mu_{1}\cdots \mu_{n} \in \Sigma$. For each $i \in \{ 1,\ldots,n \}$, since the function $\Phi_{\mu_{i}\cdots \mu_{n}}f_{i}$ is constant, it follows that \begin{displaymath}
	\Phi_{\mu}(f_{i}) \, = \,  \Phi_{\mu_{1}\cdots\mu_{i-1}}(\Phi_{\mu_{i}\cdots \mu_{n}}f_{i}) \, = \, \Phi_{\mu_{i}\cdots \mu_{n}}f_{i} ,
\end{displaymath} wherefore $\Phi_{\mu}(f_{i})$ is constant, too. Consequently, $\mu (f_{i} \circ \lambda_{g}) = (\Phi_{\mu}f_{i})(g) = (\Phi_{\mu}f_{i})(e) = \mu (f_{i})$ for all $g \in G$ and $i \in \{ 1,\ldots,n \}$, as desired. \end{proof}

The subsequent corollary is a well-known result due to Granirer and Lau~\cite{GranirerLau}.

\begin{cor}[\cite{GranirerLau}]\label{corollary:lau} Let $G$ be a topological group. The following hold. \begin{enumerate}
	\item[$(1)$]  $G$ is amenable if and only if \begin{displaymath}
	\forall f \in \mathrm{RUCB}(G) \ \exists \mu \in \mathrm{M}(G) \ \forall g \in G \colon \quad \mu (f \circ \lambda_{g}) \, = \, \mu (f) .
\end{displaymath}
	\item[$(2)$]  $G$ is extremely amenable if and only if \begin{displaymath}
		\forall f \in \mathrm{RUCB}(G) \ \exists \mu \in \mathrm{S}(G) \ \forall g \in G \colon \quad \mu (f \circ \lambda_{g}) \, = \, \mu (f) .
\end{displaymath}
\end{enumerate} \end{cor}

\begin{proof} Both statements are immediate consequences of Lemma~\ref{lemma:amenability}: the first follows for $\Sigma = \mathrm{M}(G)$ and $H = \mathrm{RUCB}(G)$, the second for $\Sigma = \mathrm{S}(G)$ and $H = \mathrm{RUCB}(G)$. \end{proof}

As shown by Moore~\cite{moore}, for a discrete group $G$, one may replace $\mathrm{RUCB}(G) = \ell^{\infty}(G)$ in Corollary~\ref{corollary:lau}(1) by the set of characteristic functions of subsets of $G$. A topological version of Moore's result, in terms of two-element uniform coverings, has been established in~\cite{SchneiderThom17}. For our present purposes, however, we will need a different refinement of Corollary~\ref{corollary:lau}(1).

\begin{cor}\label{corollary:amenability.test} Let $G$ be a topological group. Assume that $H \subseteq \mathrm{RUCB}(G)$ is introverted and generates a dense linear subspace of $\mathrm{RUCB}(G)$. Then $G$ is amenable if and only if \begin{displaymath}
	\forall f \in H \ \exists \mu \in \mathrm{M}(G) \ \forall g \in G \colon \quad \mu (f \circ \lambda_{g}) \, = \, \mu (f) .
\end{displaymath} \end{cor}

\begin{proof} Evidently, the former implies the latter. Let us prove the converse implication. By Lemma~\ref{lemma:amenability}, our hypothesis entails that there exists $\mu \in \mathrm{M}(G)$ such that $\mu (f \circ \lambda_{g}) = \mu (f)$ for all $f \in H$ and $g \in G$. Since $\mu$ is a continuous linear form and $H$ generates a dense linear subspace of $\mathrm{RUCB}(G)$, it follows that $\mu (f \circ \lambda_{g}) = \mu (f)$ for all $f \in \mathrm{RUCB}(G)$ and $g \in G$. \end{proof}

\section{Poisson boundary and Furstenberg's conjecture}\label{section:poisson}

In this section we consider harmonic functions and Poisson boundaries for random walks on topological groups and prove a general topological version of Furstenberg's conjecture.

\begin{definition}\label{definition:harmonicity} Let $G$ be a topological group and let $\mu \in \mathrm{M}(G)$. The elements of \begin{displaymath}
	\mathrm{H}_{\mu}(G) \, \defeq \, \{ f \in \mathrm{RUCB}(G) \mid f = \Phi_{\mu}f \}
\end{displaymath} will be called \emph{$\mu$-harmonic} functions on $G$. \end{definition}

Of course, for any topological group $G$ and any $\mu \in \mathrm{M}(G)$, the collection $\mathrm{H}_{\mu}(G)$ constitutes a $\Vert \cdot \Vert_{\infty}$-closed linear subspace of $\mathrm{RUCB}(G)$ containing the set of constant functions and being closed under complex conjugation. The subsequent Proposition~\ref{proposition:poisson.algebra} and Proposition~\ref{proposition:poisson.formula} are simple variations of results due to Prunaru~\cite{prunaru} extending earlier work Furstenberg~\cite{furstenberg63a} and many others~\cite{zimmer,paterson,dokken,DokkenEllis,willis}.

\begin{prop}\label{proposition:poisson.algebra} Let $G$ be a topological group and let $\mu \in \Lambda (\mathrm{M}(G))$. Then $\mathrm{H}_{\mu}(G)$ constitutes a commutative unital $C^{\ast}$-algebra with respect to the multiplication given by \begin{displaymath}
	\left(f_{1} \cdot_{\mu} f_{2}\right)\! (g) \, \defeq \, \lim\nolimits_{n \to \infty} \Phi_{\mu^{n}}(f_{1}f_{2}) (g) \qquad \left(f_{1},f_{2} \in \mathrm{H}_{\mu}(G), \, g \in G\right) .
\end{displaymath} \end{prop}

\begin{proof} The following argument is due to Prunaru~\cite{prunaru} and will be included only for the sake of convenience. By Lemma~\ref{lemma:convolution}, $(\Phi_{\mu^{n}})_{n \in \mathbb{N}}$ is a sequence of linear contractions, which readily implies that the linear subspace \begin{displaymath}
	\mathrm{C}_{\mu}(G) \, \defeq \, \left\{ f \in \mathrm{RUCB}(G) \left\vert \, (\Phi_{\mu^{n}}f)_{n \in \mathbb{N}} \text{ pointwise convergent} \right\} \right.
\end{displaymath} is $\Vert \cdot \Vert_{\infty}$-closed in $\mathrm{RUCB}(G)$. To prove this, let $f \in \mathrm{RUCB}(G)$ and $(f_{k})_{k \in \mathbb{N}} \in \mathrm{C}_{\mu}(G)^{\mathbb{N}}$ such that $\Vert f -f_{k} \Vert_{\infty} \longrightarrow 0$ as $k \to \infty$. Let $g \in G$. For every $\epsilon > 0$, we find $k \in \mathbb{N}$ with $\Vert f - f_{k} \Vert_{\infty} \leq \tfrac{\epsilon}{3}$ and then $\ell \in \mathbb{N}$ with $\sup_{m,n \in \mathbb{N}_{\geq \ell}} \vert (\Phi_{\mu^{m}} f_{k})(g) - (\Phi_{\mu^{n}} f_{k})(g) \vert \leq \tfrac{\epsilon}{3}$, which entails that \begin{displaymath}
	\vert (\Phi_{\mu^{m}} f)(g) - (\Phi_{\mu^{n}} f)(g) \vert \, \leq \, \Vert f -f_{k} \Vert_{\infty} + \vert (\Phi_{\mu^{m}} f_{k})(g) - (\Phi_{\mu^{n}} f_{k})(g) \vert + \Vert f_{k} - f \Vert_{\infty} \, \leq \, \epsilon 
\end{displaymath} for all $m,n \in \mathbb{N}_{\geq \ell}$. This shows that $((\Phi_{\mu^{n}}f)(g))_{n \in \mathbb{N}} \in \mathbb{C}^{\mathbb{N}}$ is a Cauchy sequence and therefore convergent in $\mathbb{C}$. Thus, $f \in \mathrm{C}_{\mu}(G)$ indeed. What is more, for each $n \in \mathbb{N}$, the operator $\Phi_{\mu^{n}}$ commutes with complex conjugation, whence $\mathrm{C}_{\mu}(G)$ is closed with respect to the latter, too. Furthermore, $\mathrm{H}_{\mu}(G) \subseteq \mathrm{C}_{\mu}(G)$. Thanks to Lemma~\ref{lemma:introverted}, for each $f \in \mathrm{C}_{\mu}(G)$, the pointwise limit $\pi_{\mu}(f) \defeq \lim_{n \to \infty} \Phi_{\mu^{n}}f$ belongs to $\mathrm{RUCB}(G)$. Once more, since $(\Phi_{\mu^{n}})_{n \in \mathbb{N}}$ is a sequence of linear contractions, $\pi_{\mu} \colon \mathrm{C}_{\mu}(G) \to \mathrm{RUCB}(G)$ is a linear contraction as well. Evidently, if~$f \in \mathrm{H}_{\mu}(G)$, then $\pi_{\mu}(f) = f$. We now claim that $\pi_{\mu}(\mathrm{C}_{\mu}(G)) \subseteq \mathrm{H}_{\mu}(G)$. In order to prove this inclusion, let us consider the weak-${}^{\ast}$ closed set \begin{displaymath}
	A \, \defeq \, \bigcap\nolimits_{m \in \mathbb{N}} \overline{\{ \mu^{n} \mid n \in \mathbb{N}_{\geq m} \}} \, \subseteq \, \mathrm{M}(G) ,
\end{displaymath} which is non-empty by weak-${}^{\ast}$ compactness of $\mathrm{M}(G)$. As $\mu \in \Lambda (\mathrm{M}(G))$, it follows that \begin{displaymath}
	\mu A \, \subseteq \, \bigcap\nolimits_{m \in \mathbb{N}} \overline{\{ \mu^{n} \mid n \in \mathbb{N}_{\geq m+1} \}} \, = \, A .
\end{displaymath} Now, let $f \in \mathrm{C}_{\mu}(G)$. We wish to show that $\Phi_{\mu}(\pi_{\mu}(f)) = \pi_{\mu}(f)$. Note that, for every $g \in G$, continuity of the map $\mathrm{M}(G) \to \mathbb{R}, \, \nu \mapsto \nu(f \circ \lambda_{g})$ implies that \begin{align*}
	\{ (\Phi_{\nu}f)(g) \mid \nu \in A \} \, &= \, \{ \nu (f \circ \lambda_{g}) \mid \nu \in A \} \, \subseteq \, \bigcap\nolimits_{m \in \mathbb{N}} \overline{\{ \mu^{n}(f \circ \lambda_{g}) \mid n \in \mathbb{N}_{\geq m} \}} \\
	& = \, \bigcap\nolimits_{m \in \mathbb{N}} \overline{\{ (\Phi_{\mu^{n}}f)(g) \mid n \in \mathbb{N}_{\geq m} \}} \, = \, \{ \pi_{\mu}(f)(g) \} .
\end{align*} That is, $\Phi_{\nu}f = \pi_{\mu}(f)$ for all $\nu \in A$. Hence, picking any $\nu \in A$ and using the fact that $\mu \nu \in A$, we conclude that \begin{displaymath}
	\Phi_{\mu}(\pi_{\mu}(f)) \, = \, \Phi_{\mu}(\Phi_{\nu}f) \, \stackrel{\ref{lemma:convolution}(4)}{=} \, \Phi_{\mu \nu}(f) \, = \, \pi_{\mu}(f) .
\end{displaymath} So, $\pi_{\mu}(f) \in \mathrm{H}_{\mu}(G)$ as desired. It follows that $\pi_{\mu}$ is idempotent and that $\pi_{\mu}(\mathrm{C}_{\mu}(G)) = \mathrm{H}_{\mu}(G)$.

\textit{Claim 1.} If $h \in \mathrm{H}_{\mu}(G)$, then $\vert h \vert^{2} \in \mathrm{C}_{\mu}(G)$ and \begin{displaymath}
	\forall f \in \mathrm{RUCB}(G) \colon \quad \left(\vert h \vert^{2} - \pi_{\mu}\! \left(\vert h \vert^{2}\right)\right) \! f \, \in \, \ker (\pi_{\mu})
\end{displaymath}

\textit{Proof of Claim~1.} Let $h \in \mathrm{H}_{\mu}(G)$. For each $g \in G$, since $g\mu$ is a positive linear functional on the $C^{\ast}$-algebra $\mathrm{RUCB}(G)$, the Cauchy--Schwarz inequality asserts that \begin{displaymath}
	\vert h \vert^{2}(g) \, = \, \vert {\Phi_{\mu}h} \vert^{2}(g) \, = \, \vert (g\mu) (\overline{\mathbf{1}}h)\vert^{2} \, \leq \, (g\mu)\!\left(\vert \mathbf{1} \vert^{2}\right) \cdot (g\mu)\!\left(\vert h \vert^{2}\right) \, = \, \Phi_{\mu}\!\left(\vert h \vert^{2}\right)\!(g) .
\end{displaymath} By positivity of $\Phi_{\mu}$, this entails that the $\Vert \cdot \Vert_{\infty}$-bounded sequence $\left(\Phi_{\mu^{n}}\!\left(\vert h \vert^{2}\right)\right)_{n \in \mathbb{N}}$ is \mbox{increasing}, thus pointwise convergent to the function $\sup_{n \in \mathbb{N}} \Phi_{\mu^{n}}\!\left(\vert h \vert^{2}\right)$. This shows that $\vert h \vert^{2} \in \mathrm{C}_{\mu}(G)$ and $\pi_{\mu}\!\left(\vert h \vert^{2}\right) = \sup_{n \in \mathbb{N}} \Phi_{\mu^{n}}\!\left(\vert h \vert^{2}\right)$. In particular, $\pi_{\mu}\!\left( \vert h \vert^{2} \right) - \vert h \vert^{2} \geq \mathbf{0}$. Note that \begin{displaymath}
	\Phi_{\mu^{n}}\!\left( \pi_{\mu}\!\left( \vert h \vert^{2} \right) - \vert h \vert^{2} \right) \! (g) \ \longrightarrow \ \pi_{\mu} \!\left( \pi_{\mu}\!\left( \vert h \vert^{2} \right) - \vert h \vert^{2} \right) \! (g) \, = \, 0 \qquad (n \longrightarrow \infty) 
\end{displaymath} for all $g \in G$. For every $f \in \mathrm{RUCB}(G)$ with $f \geq \mathbf{0}$, since \begin{displaymath}
	\left\lvert \Phi_{\mu^{n}}\!\left( \left( \vert h \vert^{2} - \pi_{\mu}\!\left( \vert h \vert^{2} \right) \right) \! f \right) \right\rvert \, \leq \, \Vert f \Vert_{\infty} \Phi_{\mu^{n}}\!\left( \pi_{\mu}\!\left( \vert h \vert^{2} \right) - \vert h \vert^{2} \right) 
\end{displaymath} due to positivity of the linear operators $(\Phi_{\mu^{n}})_{n \in \mathbb{N}}$, it follows that $\left( \Phi_{\mu^{n}}\!\left( \left( \vert h \vert^{2} - \pi_{\mu}\!\left( \vert h \vert^{2} \right) \right) \! f \right) \right)_{n \in \mathbb{N}}$ converges pointwise to $\mathbf{0}$, i.e., $\left(\vert h \vert^{2} - \pi_{\mu}\! \left(\vert h \vert^{2}\right) \right) \! f \in \ker (\pi_{\mu})$. By linearity of $\pi_{\mu}$, this readily implies that $\left(\vert h \vert^{2} - \pi_{\mu}\! \left(\vert h \vert^{2}\right)\right) \! f \in \ker (\pi_{\mu})$ for every $f \in \mathrm{RUCB}(G)$. \hfill $\qed_{\text{Claim\! 1}}$

\textit{Claim 2.} If $h_{1},h_{2} \in \mathrm{H}_{\mu}(G)$, then $h_{1}h_{2} \in \mathrm{C}_{\mu}(G)$ and \begin{displaymath}
	\forall f \in \mathrm{RUCB}(G) \colon \quad (h_{1}h_{2} - \pi_{\mu}(h_{1}h_{2})) f \, \in \, \ker (\pi_{\mu})
\end{displaymath}

\textit{Proof of Claim~2.} Consider any $h_{1},h_{2} \in \mathrm{H}_{\mu}(G)$. A straightforward computation shows that $h_{1}h_{2} = \tfrac{1}{4}\sum_{n=0}^{3} i^{n}\vert g_{n} \vert^{2}$, where $g_{n} \defeq h_{1} + i^{n}\overline{h_{2}}$ for $n \in \{ 0,1,2,3 \}$. Since $g_{n} \in \mathrm{H}_{\mu}(G)$ for each $n \in \{ 0,1,2,3 \}$, Claim~1 implies that $h_{1}h_{2} \in \mathrm{C}_{\mu}(G)$ and moreover \begin{displaymath}
	(h_{1}h_{2} - \pi_{\mu}(h_{1}h_{2})) f \, = \, \tfrac{1}{4}\sum\nolimits_{n=0}^{3} i^{n}\!\left( \left( \vert g_{n} \vert^{2} - \pi_{\mu}\!\left( \vert g_{n} \vert^{2} \right) \right) \! f \right) \, \in \, \ker (\pi_{\mu})
\end{displaymath} for every $f \in \mathrm{RUCB}(G)$. \hfill $\qed_{\text{Claim\! 2}}$

Henceforth, let us denote by $\mathrm{J}_{\mu}(G)$ the $\Vert \cdot \Vert_{\infty}$-closed ideal of $\mathrm{RUCB}(G)$ generated by the subset $\{ h_{1}h_{2} - \pi_{\mu}(h_{1}h_{2}) \mid h_{1},h_{2} \in \mathrm{H}_{\mu}(G) \}$. By Claim~2, $\mathrm{J}_{\mu}(G)$ is contained in $\ker (\pi_{\mu})$.

\textit{Claim 3.} If $n \in \mathbb{N}_{\geq 1}$ and $h_{1},\ldots,h_{n} \in \mathrm{H}_{\mu}(G)$, then $h_{1}\cdots h_{n} \in \mathrm{C}_{\mu}(G)$ and \begin{displaymath}
	h_{1}\cdots h_{n} - \pi_{\mu} (h_{1}\cdots h_{n}) \, \in \, \mathrm{J}_{\mu}(G) .
\end{displaymath}

\textit{Proof of Claim~3.} The proof proceeds by induction. Clearly, if $n=1$, then the statement is trivial. Moreover, if $n=2$, then the desired conclusion follows from Claim~2 and the definition of $\mathrm{J}_{\mu}(G)$. For the inductive step, let $n \in \mathbb{N}_{\geq 2}$ such that the assertion of Claim~3 is valid. Let $h_{1},\ldots,h_{n+1} \in \mathrm{H}_{\mu}(G)$ and $f \defeq h_{1}\cdots h_{n+1}$. Then $f_{1} \defeq (h_{1}h_{2} - \pi_{\mu}(h_{1}h_{2}))h_{3}\cdots h_{n+1} \in \mathrm{J}_{\mu}(G)$ and $f_{1} \in \ker (\pi_{\mu})$ by Claim~2. Furthermore, since $\pi_{\mu}(h_{1}h_{2}) \in \mathrm{H}_{\mu}(G)$, our induction hypothesis asserts that $f_{2} \defeq \pi_{\mu}(h_{1}h_{2})h_{3}\cdots h_{n+1} \in \mathrm{C}_{\mu}(G)$ and $f_{2} - \pi_{\mu}(f_{2}) \in \mathrm{J}_{\mu}(G)$. Consequently, \begin{align*}
	& f \, = \, f_{1} + f_{2} \, \in \, \mathrm{C}_{\mu}(G) , & f - \pi_{\mu}(f) \, = \, f_{1} + f_{2} - \pi_{\mu}(f_{2}) \, \in \, \mathrm{J}_{\mu}(G) .
\end{align*} This completes the induction and hence proves the claim. \hfill $\qed_{\text{Claim\! 3}}$

Let us denote by $\mathrm{A}_{\mu}(G)$ the $\Vert \cdot \Vert_{\infty}$-closed subalgebra of $\mathrm{RUCB}(G)$ generated by $\mathrm{H}_{\mu}(G)$. Since $\mathrm{H}_{\mu}(G)$ is closed under complex conjugation, $\mathrm{A}_{\mu}(G)$ is a $C^{\ast}$-subalgebra of $\mathrm{RUCB}(G)$. By Claim~3, $\mathrm{A}_{\mu}(G) \subseteq \mathrm{C}_{\mu}(G)$. As $\pi_{\mu}$ is an idempotent linear contraction, it follows that \begin{displaymath}
	\Vert h \Vert_{\infty} \, = \, \inf \{ \Vert f \Vert_{\infty} \mid f \in \mathrm{A}_{\mu}(G), \, \pi_{\mu}(f) = h \}
\end{displaymath} for all $h \in \mathrm{H}_{\mu}(G)$, and therefore $\mathrm{A}_{\mu}(G)/\ker\!\left( \pi_{\mu}\vert_{\mathrm{A}_{\mu}(G)} \right) \to \mathrm{H}_{\mu}(G), \, f + \ker\!\left( \pi_{\mu}\vert_{\mathrm{A}_{\mu}(G)} \right) \mapsto \pi_{\mu}(f)$ is an isometric isomorphism of the respective Banach spaces. Moreover, Claim~3 asserts that $f - \pi_{\mu}(f) \in \mathrm{J}_{\mu}(G)$ for every $f \in \mathrm{A}_{\mu}(G)$. Since, again, $\pi_{\mu}$ is idempotent and $\mathrm{J}_{\mu}(G) \subseteq \ker (\pi_{\mu})$ due to Claim~2, we conclude that \begin{displaymath}
	\ker \!\left( \pi_{\mu}\vert_{\mathrm{A}_{\mu}(G)} \right) \, = \, \{ f - \pi_{\mu}(f) \mid f \in \mathrm{A}_{\mu}(G) \} \, = \, \mathrm{J}_{\mu}(G) \cap \mathrm{A}_{\mu}(G) .
\end{displaymath} In particular, $\ker \!\left( \pi_{\mu}\vert_{\mathrm{A}_{\mu}(G)} \right)$ is an ideal of $\mathrm{A}_{\mu}(G)$, and thus $\mathrm{H}_{\mu}(G) \cong \mathrm{A}_{\mu}(G)/\ker\!\left( \pi_{\mu}\vert_{\mathrm{A}_{\mu}(G)} \right)$ constitutes a $C^{\ast}$-algebra with respect to complex conjugation and the multiplication given by \begin{displaymath}
	h_{1} \cdot_{\mu} h_{2} \, \defeq \, \pi_{\mu}(h_{1}h_{2}) \, = \, \lim\nolimits_{n \to \infty} \Phi_{\mu^{n}}(h_{1}h_{2}) \qquad \left(h_{1},h_{2} \in \mathrm{H}_{\mu}(G)\right) .
\end{displaymath} Evidently, the $C^{\ast}$-algebra $\mathrm{H}_{\mu}(G)$ is commutative and unital. \end{proof}

\begin{definition}\label{definition:poisson.boundary} Let $G$ be a topological group and let $\mu \in \Lambda (\mathrm{M}(G))$. The \emph{Poisson boundary} of $(G,\mu)$, denoted by $\Pi_{\mu} (G)$, is defined to be the Gel'fand spectrum of the commutative unital $C^{\ast}$-algebra $\mathrm{H}_{\mu}(G)$, i.e., the compact Hausdorff space of ${}^{\ast}$-homomorphisms from $\mathrm{H}_{\mu}(G)$ to $\mathbb{C}$, endowed with the weak-${}^{\ast}$ topology. \end{definition}

The next proposition provides an integral representation of harmonic functions via Poisson boundaries as introduced above. Let $G$ be a topological group and let $\mu \in \Lambda (\mathrm{M}(G))$. It follows from Lemma~\ref{lemma:convolution}(2) (and Proposition~\ref{proposition:poisson.algebra}) that $\mathrm{H}_{\mu}(G) \to \mathrm{H}_{\mu}(G), \, h \mapsto h \circ \lambda_{g}$ is a well-defined $C^{\ast}$-automorphism for every $g \in G$. In turn, $G$ admits an action on $\Pi_{\mu}(G)$ given by \begin{displaymath}
	(g\xi)(h) \, \defeq \, \xi (h \circ \lambda_{g}) \qquad \left(g \in G, \, \xi \in \Pi_{\mu}(G), \, h \in \mathrm{H}_{\mu}(G)\right) ,
\end{displaymath} which is easily seen to be continuous. For every $f \in \mathrm{H}_{\mu}(G)$, let $\psi_{\mu}(f) \colon \Pi_{\mu}(G) \to \mathbb{C}, \, \xi \mapsto \xi (f)$. Due to Gel'fand duality, $\psi_{\mu} \colon \mathrm{H}_{\mu}(G) \to \mathrm{C}(\Pi_{\mu}(G))$ constitutes an isomorphism of $C^{\ast}$-algebras. Since $\mathrm{H}_{\mu}(G) \to \mathbb{C}, \, h \mapsto h(e)$ is a positive unital linear functional on the $C^{\ast}$-algebra $\mathrm{H}_{\mu}(G)$, the Riesz--Markov--Kakutani representation theorem asserts that there exists a unique regular Borel probability measure $\hat{\mu}$ on $\Pi_{\mu}(G)$ such that $\int \psi_{\mu}(h) \, d\hat{\mu} = h(e)$ for all $h \in \mathrm{H}_{\mu}(G)$.

\begin{prop}[Poisson formula]\label{proposition:poisson.formula} Let $G$ be a topological group and let $\mu \in \Lambda (\mathrm{M}(G))$. For all $h \in \mathrm{H}_{\mu}(G)$ and $g \in G$, \begin{displaymath}
	h(g) \, = \, \int \psi_{\mu}(h)(g\xi) \, d\hat{\mu}(\xi) .
\end{displaymath} \end{prop}

\begin{proof} For all $h \in \mathrm{H}_{\mu}(G)$ and $g \in G$, \begin{align*}
	\int \psi_{\mu}(h)(g\xi) \, d\hat{\mu}(\xi) \, &= \, \int (g\xi)(h) \, d\hat{\mu}(\xi) \, = \, \int \xi(h \circ \lambda_{g}) \, d\hat{\mu}(\xi) \\
	& = \, \int \psi_{\mu}(h \circ \lambda_{g})(\xi) \, d\hat{\mu}(\xi) \, = \, (h \circ \lambda_{g})(e) \, = \, h(g) . \qedhere
\end{align*} \end{proof}

Let us study the connection between the amenability of topological groups and the structure of their Poisson boundaries. By Lemma~\ref{lemma:semigroup.closure}, if $G$ is a topological group and $\mu \in \Lambda(\mathrm{M}(G))$, then the weak-${}^{\ast}$ closed convex hull $\Sigma_{\mu} (G) \defeq \convc \{ \mu^{n} \mid n \geq 1 \}$ is a subsemigroup of $\mathrm{M}(G)$.

\begin{lem}\label{lemma:boundary.triviality} Let $G$ be a topological group, let $\mu \in \Lambda(\mathrm{M}(G))$, and let $H \subseteq \mathrm{RUCB}(G)$ be introverted. The following are equivalent. \begin{enumerate}
	\item[$(1)$] $\exists \nu \in \Sigma_{\mu}(G) \ \forall f \in H \ \forall g \in G \colon \ \nu (f \circ \lambda_{g}) = \nu (f)$.
	\item[$(2)$] $\mathrm{H}_{\mu}(G) \cap H \subseteq \mathbb{C}$.		
\end{enumerate} \end{lem}

\begin{proof} (1)$\Longrightarrow$(2). Let $h \in \mathrm{H}_{\mu}(G)$. By Lemma~\ref{lemma:introverted}, $\Psi_{h}^{-1}(\{ h \}) = \{ \nu \in \mathrm{M}(G) \mid \Phi_{\nu}h = h \}$ forms a closed subset of $\mathrm{M}(G)$. Since moreover $\Psi_{h}^{-1}(\{ h \})$ is a convex subsemigroup of $\mathrm{M}(G)$ containing $\mu$, it follows that $\Sigma_{\mu}(G) \subseteq \Psi_{h}^{-1}(\{ h \})$, i.e., $\Phi_{\nu}h = h$ for every $\nu \in \Sigma_{\mu}(G)$. Hence, if $h \in H$ and there exists $\nu \in \Sigma_{\mu}(G)$ with $\nu (f \circ \lambda_{g}) = \nu (f)$ for all $g \in G$, then \begin{displaymath}
	h(g) = (\Phi_{\nu}h)(g) = \nu (h \circ \lambda_{g}) = \nu (h) = (\Phi_{\nu}h)(e) = h(e)
\end{displaymath} for all $g \in G$, that is, $h$ is constant. 
	
(2)$\Longrightarrow$(1). Since $\mu \in \Lambda(\mathrm{M}(G))$, the affine map $\Sigma_{\mu}(G) \rightarrow \Sigma_{\mu}(G), \, \nu \mapsto \mu \nu$ is continuous. By the Markov--Kakutani fixed-point theorem~\cite{markov,Kakutani38}, there exists $\nu \in \Sigma_{\mu}(G)$ with $\mu \nu = \nu$. Now, consider any $f \in H$. Then $\Phi_{\nu}f \in H$, due to $H$ being introverted. Moreover, since $\Phi_{\mu}(\Phi_{\nu}f) = \Phi_{\mu \nu}f = \Phi_{\nu}f$ by Lemma~\ref{lemma:convolution}(4), the function $\Phi_{\nu}f$ is $\mu$-harmonic, hence constant by assumption. Thus, $\nu (f \circ \lambda_{g}) = (\Phi_{\nu}f)(g) = (\Phi_{\nu}f)(e) = \nu (f)$ for every $g \in G$.\end{proof}

In the light of Remark~\ref{remark:measures.vs.functionals}, the following proposition particularly applies to any regular Borel probability measure on a topological group (Proposition~\ref{proposition:topological.centre}) as well as to any arbitrary Borel probability measure on a separable topological group (Proposition~\ref{proposition:topological.centre.pachl}).

\begin{prop}\label{proposition:trivial.boundary} Let $G$ be a topological group and let $\mu \in \Lambda (\mathrm{M}(G))$. Then the following are equivalent. \begin{enumerate}
	\item[$(1)$] The Poisson boundary $\Pi_{\mu} (G)$ is trivial, i.e., a singleton.
	\item[$(2)$] $\mathrm{H}_{\mu}(G) = \mathbb{C}$.
	\item[$(3)$] There is $\nu \in \Sigma_{\mu}(G)$ such that $\nu (f \circ \lambda_{g}) = \nu (f)$ for all $f \in \mathrm{RUCB}(G)$ and $g \in G$.
\end{enumerate} \end{prop}

\begin{proof} The equivalence of~(1) and~(2) is an immediate consequence of Gel'fand duality, while the equivalence of~(2) and~(3) follows by Lemma~\ref{lemma:boundary.triviality} for $H = \mathrm{RUCB}(G)$. \end{proof}

The subsequent result generalizes work of Kaimanovich--Vershik~\cite[Theorem~4.3]{KaimanovichVershik} and Rosenblatt~\cite[Theorem~1.10]{rosenblatt}. The proof given below follows closely the lines of Kaimanovich and Vershik~\cite[Proof of Theorem~4.3]{KaimanovichVershik}.

\begin{thm}\label{theorem:main} Let $G$ be a second-countable topological group. Then the following statements are equivalent. \begin{enumerate}
	\item[$(1)$] $G$ is amenable.
	\item[$(2)$] $G$ admits a fully supported, regular Borel probability measure $\mu$ such that $(\mu^{n})_{n \in \mathbb{N}}$ UEB-converges to invariance over $G$.
	\item[$(3)$] $G$ admits a Borel probability measure $\mu$ such that $(\mu^{n})_{n \in \mathbb{N}}$ UEB-converges to invariance over $G$.
\end{enumerate} \end{thm}

\begin{proof} (2)$\Longrightarrow$(3). This is trivial.
	
(3)$\Longrightarrow$(1). This is due to Remark~\ref{remark:convergence.to.invariance}.
	
(1)$\Longrightarrow$(2). Since $G$ is second-countable, the Birkhoff--Kakutani theorem~\cite{birkhoff,Kakutani36} asserts that $G$ admits a right-invariant metric $d$ generating the topology of $G$, and furthermore we find an increasing sequence of finite subsets \begin{displaymath}
	\{ e \} \eqdef S_{0} \, \subseteq \, S_{1} \, \subseteq \, \ldots \, \subseteq \, S_{m} \, \subseteq \, S_{m+1} \, \subseteq \, \ldots \, \subseteq \, G
\end{displaymath} such that $S \defeq \bigcup \{ S_{m} \mid m \in \mathbb{N}\}$ is dense in $G$. Choose a sequence $(\tau_{m})_{m \in \mathbb{N}}$ of positive reals so that $\sum_{m \in \mathbb{N}} \tau_{m} = 1$. For each $m \in \mathbb{N}_{\geq 1}$, we pick $n_{m} \in \mathbb{N}_{\geq 1}$ such that $(\tau_{0}+\ldots+\tau_{m-1})^{n_{m}} < \frac{1}{m}$. Thanks to Corollary~\ref{corollary:topological.day}, we may recursively choose a sequence $(\alpha_{m})_{m \in \mathbb{N}}$ of finitely supported, regular Borel probability measures on $G$ such that $S_{0} \subseteq \spt (\alpha_{0})$ and, for each $m \in \mathbb{N}_{\geq 1}$, \begin{itemize}
	\item[(i)] $\mathrm{p}_{d}(g\alpha_{m}-\alpha_{m}) < \frac{1}{m}$ for all $g \in S_{m} \cup (\spt \alpha_{m-1})^{n_{m}}$, and \smallskip
	\item[(ii)] $S_{m} \cup (\spt \alpha_{m-1})^{n_{m}} \subseteq \, \spt (\alpha_{m})$. 
\end{itemize} Consider the regular Borel probability measure $\mu \defeq \sum_{m \in \mathbb{N}} \tau_{m}\alpha_{m}$ on $G$. Then $\spt (\mu) = G$, since $\spt (\mu)$ is a closed subset of $G$ containing $\spt (\alpha_{m}) \supseteq S_{m}$ for all $m \in \mathbb{N}$. We will show that $(\mu^{n})_{n \in \mathbb{N}}$ UEB-converges to invariance over $G$. Our proof proceeds in three steps.

\textit{Claim 1.} For all $m \in \mathbb{N}_{\geq 1}$, $k \in \mathbb{N}^{n_{m}}\setminus \{ 0,\ldots,m-1 \}^{n_{m}}$ and $g \in S_{m-1}$, \begin{displaymath}
	\mathrm{p}_{d}\! \left(g\alpha_{k_{1}} \cdots \alpha_{k_{n_{m}}} - \alpha_{k_{1}} \cdots \alpha_{k_{n_{m}}}\right) \, \leq \, \tfrac{2}{m} .
\end{displaymath}

\textit{Proof of Claim~1.} Let $m \in \mathbb{N}_{\geq 1}$ and $\ell\defeq n_{m}$. Let $g \in S_{m-1}$ and $k \in \mathbb{N}^{\ell}\setminus \{ 0,\ldots,m-1 \}^{\ell}$, and put $j \defeq \min \{ i \in \{ 1,\ldots,\ell \} \mid k_{i} \geq m \}$. For \begin{displaymath}
	\theta \defeq \alpha_{k_{1}} \cdots \alpha_{k_{\ell}} , \qquad \theta_{1} \defeq \alpha_{k_{1}} \cdots \alpha_{k_{j-1}} , \qquad \theta_{2} \defeq \alpha_{k_{j+1}} \cdots \alpha_{k_{\ell}} , 
\end{displaymath} we note that $\theta = \theta_{1} \alpha_{k_{j}} \theta_{2}$. For each $i \in \{ 1,\ldots,j-1 \}$, the definition of $j$ implies that $k_{i} < m$ and therefore $\spt (\alpha_{k_{i}}) \subseteq \spt (\alpha_{m-1})$ by~(ii). Since $j \leq \ell$, we have $\spt (\theta_{1}) \subseteq (\spt \alpha_{m-1})^{\ell - 1}$. Also, $g \in \spt (\alpha_{m-1})$ by~(ii), and so $\spt (g\theta_{1}) \subseteq (\spt \alpha_{m-1})^{\ell}$. Thus, assertion~(i) implies that \begin{align*}
	& \mathrm{p}_{d}\!\left(\alpha_{k_{j}} - \theta_{1}\alpha_{k_{j}}\right) \, \leq \, \sum\nolimits_{h \in \spt (\theta_{1})} \theta_{1}(\{ h \}) \mathrm{p}_{d}\!\left(\alpha_{k_{j}} - h\alpha_{k_{j}}\right) \, \leq \, \tfrac{1}{m} , \\
	& \mathrm{p}_{d}\!\left(\alpha_{k_{j}} - g\theta_{1}\alpha_{k_{j}}\right) \, \leq \, \sum\nolimits_{h \in \spt (g\theta_{1})} (g\theta_{1})(\{ h \}) \mathrm{p}_{d}\!\left(\alpha_{k_{j}} - h\alpha_{k_{j}}\right) \, \leq \, \tfrac{1}{m} .
\end{align*} Consequently, thanks to Corollary~\ref{corollary:UEB.metric}, \begin{displaymath}
	\mathrm{p}_{d}(g\theta - \theta ) \, = \, \mathrm{p}_{d}\!\left( \! \left( g\theta_{1}\alpha_{k_{j}} - \theta_{1}\alpha_{k_{j}} \right) \! \theta_{2} \right) \, \leq \, \mathrm{p}_{d}\!\left(g\theta_{1}\alpha_{k_{j}} - \theta_{1}\alpha_{k_{j}}\right) \, \leq \, \tfrac{2}{m} .
\end{displaymath} This finishes the proof of Claim~1. \hfill $\qed_{\text{Claim\! 1}}$

\textit{Claim 2.} For every $m \in \mathbb{N}_{\geq 1}$ and $g \in S_{m-1}$, \begin{displaymath}
	\mathrm{p}_{d}\!\left(g\mu^{n_{m}} - \mu^{n_{m}}\right) \, < \, \tfrac{4}{m} .
\end{displaymath}

\textit{Proof of Claim~2.} Consider any $m \in \mathbb{N}_{\geq 1}$ and $g \in S_{m-1}$. We will abbreviate $\ell\defeq n_{m}$. Noting that $\mu^{\ell} = \sum\nolimits_{k \in \mathbb{N}^{\ell}} \tau_{k_{1}}\cdots \tau_{k_{\ell}} \alpha_{k_{1}} \cdots \alpha_{k_{\ell}}$, we define \begin{displaymath}
	\nu_{1} \, \defeq \, \sum\nolimits_{k \in \{ 0,\ldots,m-1 \}^{\ell}} \tau_{k_{1}}\cdots \tau_{k_{\ell}} \alpha_{k_{1}} \cdots \alpha_{k_{\ell}}
\end{displaymath} and $\nu_{2} \defeq \mu^{\ell} - \nu_{1}$. Evidently, \begin{displaymath}
	\mathrm{p}_{d}(g\nu_{1}-\nu_{1}) \leq \mathrm{p}_{d}(g\nu_{1}) + \mathrm{p}_{d}(\nu_{1}) \leq 2 \sum\nolimits_{k \in \{ 0,\ldots,m-1 \}^{\ell}} \tau_{k_{1}}\cdots \tau_{k_{\ell}} = 2(\tau_{0} + \ldots + \tau_{m-1})^{\ell} < \tfrac{2}{m} .
\end{displaymath} Furthermore, according to Claim~1, \begin{displaymath}
	\mathrm{p}_{d}(g\nu_{2} - \nu_{2}) \, \leq \, \sum\nolimits_{k \in \mathbb{N}^{\ell}\setminus \{ 0,\ldots,m-1\}^{\ell}} \tau_{k_{1}}\cdots \tau_{k_{\ell}} \mathrm{p}_{d}\!\left(g\alpha_{k_{1}} \cdots \alpha_{k_{\ell}} - \alpha_{k_{1}} \cdots \alpha_{k_{\ell}}\right) \, \leq \, \tfrac{2}{m} .
\end{displaymath} Consequently, $\mathrm{p}_{d}\!\left(g\mu^{\ell}-\mu^{\ell}\right) \leq \mathrm{p}_{d}(g\nu_{1}-\nu_{1}) + \mathrm{p}_{d}(g\nu_{2}-\nu_{2}) \leq \frac{4}{m}$ as desired. \hfill $\qed_{\text{Claim\! 2}}$

\textit{Claim 3.} The sequence $(\mu^{n})_{n \in \mathbb{N}}$ UEB-converges to invariance over $G$.\newline
\textit{Proof of Claim 3.} Thanks to Lemma~\ref{lemma:dense.convergence.to.invariance} and $S$ being dense in $G$, it is sufficient to show that, for every $g \in S$, the sequence $(g\mu^{n}-\mu^{n})_{n \in \mathbb{N}}$ converges to $\mathbf{0} \in \mathrm{RUCB}(G)^{\ast}$ with respect to the UEB topology. For this purpose, let $g \in S$. According to Corollary~\ref{corollary:UEB.metric}, \begin{displaymath}
	\mathrm{p}_{d}\!\left(g\mu^{n+1}-\mu^{n+1}\right) \, = \, \mathrm{p}_{d}((g\mu^{n}-\mu^{n})\mu) \, \leq \, \mathrm{p}_{d}(g\mu^{n}-\mu^{n})
\end{displaymath} for all $n \in \mathbb{N}$, i.e., the sequence $(\mathrm{p}_{d}(g\mu^{n}-\mu^{n}))_{n\in \mathbb{N}}$ is decreasing. Moreover, Claim~2 gives that \begin{displaymath}
	\inf\nolimits_{n \in \mathbb{N}} \mathrm{p}_{d}(g\mu^{n}-\mu^{n}) \, \leq \, \inf\nolimits_{m \in \mathbb{N}} \mathrm{p}_{d}(g\mu^{n_{m}}-\mu^{n_{m}}) \, = \, 0 ,
\end{displaymath} and therefore $\mathrm{p}_{d}(g\mu^{n} - \mu^{n}) \longrightarrow 0$ as $n \to \infty$. Since $\{ g\mu^{n} - \mu^{n} \mid n \in \mathbb{N} \} \cup \{ \mathbf{0} \}$ is a $\Vert \cdot \Vert$-bounded subset of $\mathrm{RUCB}(G)^{\ast}$, it follows by Lemma~\ref{lemma:UEB.metric} that $(g\mu^{n}-\mu^{n})_{n \in \mathbb{N}}$ converges to $\mathbf{0} \in \mathrm{RUCB}(G)^{\ast}$ with respect to the UEB topology, as desired. \hfill $\qed_{\text{Claim\! 3}}$ \end{proof}

We deduce a general topological version of Furstenberg's conjecture~\cite{furstenberg73}, established for countable discrete groups by Kaimanovich--Vershik~\cite{KaimanovichVershik} and for second-countable locally compact groups by Rosenblatt~\cite{rosenblatt}. 

\begin{cor}\label{corollary:main} Let $G$ be a second-countable topological group. Then the following statements are equivalent. \begin{enumerate}
	\item[$(1)$] $G$ is amenable.
	\item[$(2)$] $G$ admits a fully supported, regular Borel probability measure $\mu$ such that $\Pi_{\mu}(G)$ is trivial.
	\item[$(3)$] $G$ admits a Borel probability measure $\mu$ such that $\Pi_{\mu}(G)$ is trivial.
\end{enumerate} \end{cor}

\begin{proof} (2)$\Longrightarrow$(3). This is trivial.
	
(3)$\Longrightarrow$(1). This is an immediate consequence of Proposition~\ref{proposition:topological.centre.pachl} and Proposition~\ref{proposition:trivial.boundary}.

(1)$\Longrightarrow$(2). Suppose that $G$ is amenable. By Theorem~\ref{theorem:main}, there exists a fully supported, regular Borel probability measure $\mu$ on $G$ such that $(\mu^{n})_{n \in \mathbb{N}}$ UEB-converges to invariance. Thanks to compactness, we find a weak-${}^{\ast}$ accumulation point $\nu$ of the sequence $(\mu^{n})_{n \in \mathbb{N}}$ in~$\mathrm{M}(G)$. Evidently, $\nu \in \Sigma_{\mu}(G)$. Due to Remark~\ref{remark:convergence.to.invariance}, $\nu$ is $G$-invariant. Hence, $\Pi_{\mu}(G)$ must be trivial by Proposition~\ref{proposition:topological.centre} and Proposition~\ref{proposition:trivial.boundary}. \end{proof}

\begin{remark} It would be very interesting to establish an analogue of Theorem~\ref{theorem:main} for more general topological groups. In this regard, it seems natural to consider separable or, more generally, $\omega$-bounded groups. Recall that a topological group~$G$ is \emph{$\omega$-bounded} if for every open neighborhood $U$ of the neutral element in $G$ there exists a countable subset $C \subseteq G$ such that $UC = G$. By work of Guran~\cite{guran}, a topological group is $\omega$-bounded if and only if it is isomorphic to a topological subgroup of a product of second-countable groups. This suggests employing inverse spectra techniques to investigate potential generalizations of Theorem~\ref{theorem:main}. \end{remark}

\section{Liouville actions on metric spaces}\label{section:liouville}

In this section, we turn our attention towards continuous isometric actions of topological groups on metric spaces and study the Liouville property for such actions. More precisely, given a metric space $X$, we consider the topological group $\mathrm{Iso}(X)$ of all isometric self-bijections of $X$, endowed with the topology of pointwise convergence. Of course, a continuous isometric action of a topological group $G$ upon $X$ corresponds naturally to a continuous homomorphism from $G$ into $\mathrm{Iso}(X)$.

Our first observation concerns introverted UEB sets on topological groups arising from continuous isometric actions. For this purpose, let us define ${f\!\!\upharpoonright_{x}} \colon G \to \mathbb{R}, \, g \mapsto f\!\left(g^{-1}x\right)$ for any group $G$ acting on a set $X$, any $x \in X$ and $f \in \mathbb{R}^{X}$.

\begin{lem}\label{lemma:metric.stabilizer.introverted} Let $G$ be a topological group acting continuously by isometries upon a metric space $X$. For every $x \in X$, the set $\mathrm{L}_{x}(G, X) \defeq \left\{ {f\!\!\upharpoonright_{x}} \left\vert \, f \in \mathrm{Lip}_{1}^{1}(X) \right\} \right.$ belongs to $\mathrm{RUEB}(G)$, is convex, right-translation closed, and compact w.r.t.~the topology of pointwise convergence, thus is introverted. \end{lem}

\begin{proof} Let $x \in X$. Evidently, $\mathrm{L}_{x}(G,X)$ is $\Vert \cdot \Vert_{\infty}$-bounded. To show that $\mathrm{L}_{x}(G,X)$ is right-uniformly equicontinuous, let us consider any $\epsilon > 0$. Then $U \defeq \{ g \in G \mid d_{X}(x,gx) \leq \epsilon \}$ constitutes a neighborhood of the neutral element in $G$. If $g,h \in G$ and $gh^{-1} \in U$, then \begin{displaymath}
	\vert {f\!\!\upharpoonright_{x}}(g) - {f\!\!\upharpoonright_{x}}(h) \vert \, = \, \left\lvert f\!\left(g^{-1}x\right) - f\!\left(h^{-1}x\right) \right\rvert \, \leq \, d_{X}\!\left(g^{-1}x,h^{-1}x\right) \, = \, d_{X}(x,gh^{-1}x) \, \leq \, \epsilon
\end{displaymath} for all $f \in \mathrm{Lip}_{1}^{1}(X)$. So, $\mathrm{L}_{x}(G,X) \in \mathrm{RUEB}(G)$ as desired. To conclude, we note that the~map \begin{displaymath}
	T_{x} \colon \, \mathrm{Lip}^{\infty}(X) \, \longrightarrow\,  \mathrm{RUCB}(G), \quad f \, \longmapsto \, {f\!\!\upharpoonright_{x}}
\end{displaymath} is linear and continuous with regard to the respective topologies of pointwise convergence. Consequently, $\mathrm{L}_{x}(G,X) = T_{x}\!\left( \mathrm{Lip}_{1}^{1}(X) \right)$ is convex and compact. Furthermore, for all $g \in G$ and $f \in \mathrm{Lip}_{1}^{1}(X)$, note that $ {f\!\!\upharpoonright_{x}} \circ \rho_{g} = {(f \circ \tau_{g})\!\!\upharpoonright_{x}} \in \mathrm{L}_{x}(G,X)$, where $\tau_{g} \colon X \to X, \, y \mapsto g^{-1}y $. Thus, $\mathrm{L}_{x}(G,X)$ is right-translation closed, hence introverted by Corollary~\ref{corollary:introverted}. \end{proof}

For convenience, let us recall the following well-known fact.

\begin{lem}\label{lemma:lipschitz.approximation} Let $S$ be a set and let $X$ be a metric space. Let $\phi \colon S \to X$ and $\ell, \epsilon \in \mathbb{R}_{\geq 0}$. For any bounded $f \colon S \to \mathbb{R}$, \begin{displaymath}
	\bigl( \, \forall s,t \in S \colon \ \vert f(s) - f(t) \vert \, \leq \, \ell d_{X}(\phi(s),\phi(t)) + \epsilon \, \bigr) \ \, \Longrightarrow \ \, \bigl( \, \exists F \in \mathrm{Lip}_{\ell}^{\infty}(X) \colon \ \Vert f - F \circ \phi \Vert_{\infty} \, \leq \, \epsilon \, \bigr) .
\end{displaymath} \end{lem}

\begin{proof} Let $r \defeq \Vert f \Vert_{\infty}$. Define $F \colon X \to \mathbb{R}$ by \begin{displaymath}
	F(x) \, \defeq \, \left( \inf\nolimits_{s \in S} f(s) + \ell d_{X}(\phi(s),x) \right) \wedge r \qquad (x \in X) .
\end{displaymath} Note that $F \in \mathrm{Lip}_{\ell}^{r}(X)$. Evidently, $F(\phi(s)) \leq f(s) + \ell d(\phi(s),\phi(s)) = f(\phi(s))$ for every~$s \in S$. Also, if $s \in S$, then $f(s) \leq f(t) + \ell d(\phi(t),\phi(s)) + \epsilon$ for all $t \in S$. Hence, $\Vert f - F\circ \phi \Vert_{\infty} \leq \epsilon$. \end{proof}

Given any metric space $X$ and $n \in \mathbb{N}$, let us consider the metric space $X^{n}$ carrying the usual supremum metric defined by \begin{displaymath}
	d_{X^{n}}(x,y) \, \defeq \, \sup \{ d_{X}(x_{i},y_{i}) \mid i \in \mathbb{N}, \, 1 \leq i \leq n \} \qquad \left( x,y \in X^{n} \right) .
\end{displaymath} Evidently, if $G$ is a topological group acting continuously by isometries on $X$, then $G$ admits a continuous isometric action upon $X^{n}$ given by $gx \defeq (gx_{1},\ldots,gx_{n})$ for all $g \in G$ and $x \in X^{n}$.

\begin{lem}\label{lemma:subgroup.density} Let $X$ be a metric space. If $G$ is any topological subgroup of $\mathrm{Iso}(X)$, then $\mathrm{L}(G,X) \defeq \bigcup \{ \mathrm{L}_{x}(G,X^{n}) \mid x \in X^{n}, \, n \in \mathbb{N} \}$ generates a dense linear subspace of $\mathrm{RUCB}(G)$. \end{lem}

\begin{proof} Of course, it suffices to check that $\mathrm{RUCB}(G) \cap \mathbb{R}^{G}$ is contained in the norm-closure of the linear subspace of $\mathrm{RUCB}(G)$ generated by $\mathrm{L}(G,X)$. So, let $f \colon G \to \mathbb{R}$ be bounded and right-uniformly continuous. Consider any $\epsilon > 0$. Thanks to right-uniform continuity, there exist $\delta > 0$, $n \in \mathbb{N}$, and $x \in X^{n}$ such that \begin{displaymath}
	\forall g,h \in G \colon \qquad d_{X^{n}}\!\left(g^{-1}x,h^{-1}x\right) \, = \, d_{X^{n}}\!\left(x,gh^{-1}x\right) \, \leq \, \delta \quad \Longrightarrow \quad \vert f(g) - f(h) \vert \, \leq \, \epsilon .
\end{displaymath} Let $\ell \defeq 2\delta^{-1}\Vert f \Vert_{\infty}$. We deduce that \begin{displaymath}
	\vert f(g) - f(h) \vert \, \leq \, \max \left\{ \epsilon, \, \ell d_{X^{n}}\!\left(g^{-1}x,h^{-1}x\right) \right\} \, \leq \, \ell d_{X^{n}}\!\left(g^{-1}x,h^{-1}x\right) + \epsilon
\end{displaymath} for all $g,h \in G$. Due to Lemma~\ref{lemma:lipschitz.approximation}, we find $F \in \mathrm{Lip}_{\ell}^{\infty}(X^{n})$ with $\sup_{g \in G} \left\lvert f(g) - F\!\left( g^{-1}x \right) \right\rvert \leq \epsilon$, i.e., $\Vert f - {F\!\!\upharpoonright_{x}} \Vert_{\infty} \leq \epsilon$. Of course, being a member of $\mathrm{Lip}_{\ell}^{\infty}(X^{n})$, the function $F$ is contained in the linear subspace of the real vector space $\mathbb{R}^{X^{n}}$ generated by $\mathrm{Lip}_{1}^{1}(X^{n})$, whence ${F\!\!\upharpoonright_{x}}$ belongs to the linear span of $\mathrm{L}_{x}(G,X)$ inside the real vector space $\mathrm{RUCB}(G)$. This proves that $f$ is contained in the norm-closure of the linear span of $\mathrm{L}(G,X)$ in $\mathrm{RUCB}(G)$, as desired. \end{proof}

\begin{cor}\label{corollary:isometry.groups} Let $X$ be a metric space and let $G$ be any topological subgroup of $\mathrm{Iso}(X)$. Then $G$ is amenable if and only if \begin{displaymath}
	\forall n \geq 1 \ \forall x \in X^{n} \ \forall f \in \mathrm{Lip}_{1}^{1}(X^{n}) \ \exists \mu \in \mathrm{M}(G) \ \forall g \in G \colon \quad \mu ({f\!\!\upharpoonright_{x}} \circ \lambda_{g}) = \mu ({f\!\!\upharpoonright_{x}}) .
\end{displaymath} \end{cor}

\begin{proof} As any union of introverted subsets of $\mathrm{RUCB}(G)$ will be introverted as well, Lemma~\ref{lemma:metric.stabilizer.introverted} entails that $\mathrm{L}(G,X)$ is an introverted subset of $\mathrm{RUCB}(G)$. Hence, the desired statement follows from Corollary~\ref{corollary:amenability.test} and Lemma~\ref{lemma:subgroup.density}. \end{proof}

Subsequently, we reformulate the results above in terms of the Liouville property. Let us call a Borel probability measure $\mu$ on a topological group $G$ \emph{non-degenerate} if $\spt (\mu)$ generates a dense subsemigroup of $G$.

\begin{definition}\label{definition:liouville} Let $G$ be a topological group acting continuously by isometries on a metric space $X$ and let $\mu$ be a Borel probability measure on $G$. A bounded measurable function $f \colon X \to \mathbb{R}$ will be called \emph{$\mu$-harmonic} if $f(x) = \int f(gx) \, d\mu (g)$ for all $x \in X$. The action of~$G$ on $X$ is called \emph{$\mu$-Liouville} if every $\mu$-harmonic uniformly continuous bounded real-valued function on $X$ is constant, and the action is said to be \emph{Liouville} if it is $\nu$-Liouville for some non-degenerate, regular Borel probability measure $\nu$ on $G$. \end{definition}

In the special case of a discrete group acting on a set, the Liouville property introduced above coincides with the usual one, e.g., as considered in~\cite{JuschenkoZheng}.

\begin{remark}\label{remark:liouville} Let $G$ be a topological group acting continuously by isometries on a metric space $X$ and let $\mu$ be any Borel probability measure on $G$. \begin{enumerate}
	\item[$(1)$] The action of $G$ on $X$ is $\mu$-Liouville if and only if every $\mu$-harmonic member of $\mathrm{Lip}_{1}^{1}(X)$ is constant. This is because $\mathrm{Lip}_{1}^{1}(X)$ spans a dense linear subspace in the Banach space $\mathrm{UCB}(X,\mathbb{R})$ of all uniformly continuous, bounded real-valued functions, equipped with the pointwise operations and the supremum norm, and \begin{displaymath}
					\qquad \qquad \ \mathrm{UCB}(X,\mathbb{R}) \, \longrightarrow \, \mathrm{UCB}(X,\mathbb{R}), \quad f \, \longmapsto \, \left(x \mapsto \int f(gx) \, d\mu(g)\right)
				\end{displaymath} is a bounded linear operator.
	\item[$(2)$] As $G$ acts isometrically on $X$, the set $\! \left. X/\!\!/G \defeq \left\{ \overline{Gx} \, \right\vert x \in X \right\}$ forms a partition of~$X$. Moreover, $X/\!\!/G$ admits a well-defined metric given \begin{displaymath}
		\qquad \qquad \ d_{X/\!\!/G} \! \left( \overline{Gx}, \overline{Gy} \right) \, \defeq \, \inf\nolimits_{g \in G} d_{X}(x,gy) \, = \, \inf\nolimits_{g \in G} d_{X}(gx,y) \qquad (x,y \in X) .
	\end{displaymath} For every $f \in \mathrm{UCB}(X/\!\!/G,\mathbb{R})$, the map $X \to \mathbb{R}, \, x \mapsto f(\overline{Gx})$ is a $\mu$-harmonic member of $\mathrm{UCB}(X,\mathbb{R})$. So, if the action of $G$ on $X$ is $\mu$-Liouville, then $\overline{Gx} = X$ for all $x \in X$.
\end{enumerate} \end{remark}

We will relate the Liouville property defined above to the study of harmonic functions on topological groups. Given any Borel probability measure $\mu$ on a topological group $G$, let us consider the push-forward measure $\mu^{\star} \defeq \iota_{\ast}(\mu)$ along the map $\iota \colon G \to G, \, g \mapsto g^{-1}$. 

\begin{lem}\label{lemma:isometric.action.harmonicity} Let $G$ be a topological group acting continuously by isometries upon a metric space $X$ and let $\mu$ be a Borel probability measure on $G$. A bounded measurable function $f \colon X \to \mathbb{R}$ is $\mu$-harmonic if and only if ${f\!\!\upharpoonright_{x}} \colon G \to \mathbb{R}$ is $\mu^{\star}$-harmonic for every $x \in X$. \end{lem}

\begin{proof} Let $f \in \mathbb{R}^{X}$ be bounded, measurable. As $f(gh^{-1}x) = f((hg^{-1})^{-1}x) = {f\!\!\upharpoonright_{x}}(hg^{-1})$ for all $g,h \in G$ and $x \in X$, it follows that \begin{align*}
	f \ \text{is $\mu$-harmonic} \quad & \Longleftrightarrow \quad \forall x \in X \colon \ \, f(x) = \int f(gx) \, d\mu(g) \\
	& \Longleftrightarrow \quad \forall x \in X \ \forall h \in G \colon \ \, f(h^{-1}x) = \int f(gh^{-1}x) \, d\mu(g) \\
	& \Longleftrightarrow \quad \forall x \in X \ \forall h \in G \colon \ \, {f\!\!\upharpoonright_{x}}(h) = \int {f\!\!\upharpoonright_{x}}(hg^{-1}) \, d\mu(g) \\
	& \Longleftrightarrow \quad \forall x \in X \ \forall h \in G \colon \ \, {f\!\!\upharpoonright_{x}}(h) = \int {f\!\!\upharpoonright_{x}}(hg) \, d\mu^{\star}(g) \\
	& \Longleftrightarrow \quad \forall x \in X \colon \ \, {f\!\!\upharpoonright_{x}} \ \text{is $\mu^{\star}$-harmonic}. \qedhere
\end{align*} \end{proof}

Everything is prepared to prove the following characterization of amenability of isometry groups in terms of the Liouville property for their induced actions.

\begin{thm}\label{theorem:juschenko} Let $X$ be a separable metric space. A topological subgroup $G$ of $\mathrm{Iso}(X)$ is amenable if and only if, for all $n \in \mathbb{N}$ and $x \in X^{n}$, the action of $G$ on $Gx$ is Liouville. \end{thm}

\begin{proof} ($\Longrightarrow$) As $X$ is a separable metric space, $\mathrm{Iso}(X)$ is second-countable with respect to the topology of pointwise convergence, and hence is the topological subgroup $G$. Since $G$ is amenable, Corollary~\ref{corollary:main} asserts that $G$ admits a fully supported, regular Borel probability measure $\mu$ such that $\mathrm{H}_{\mu}(G) = \mathbb{C}$. It follows that $\mu^{\star}$ is a fully supported (thus non-degenerate), regular Borel probability measure on $G$. To conclude, let $n \in \mathbb{N}$ and~$x \in X^{n}$. If $f \in \mathrm{Lip}_{1}^{1}(Gx)$ is $\mu^{\star}$-harmonic, then ${f\!\!\upharpoonright_{x}} \in \mathrm{RUCB}(G)$ will be $\mu$-harmonic by Lemma~\ref{lemma:metric.stabilizer.introverted}, thus constant by assumption, so that $f$ will be constant, too. Hence, the action of $G$ upon $Gx$ is Liouville.
	
($\Longleftarrow$) We apply Corollary~\ref{corollary:isometry.groups} to deduce amenability. To this end, let $n \in \mathbb{N}$ and $x \in X^{n}$. According to our assumption, we find a non-degenerate, regular Borel probability measure $\mu$ on $G$ such that the action of~$G$ on $Gx$ is $\mu$-Liouville, and therefore $\mathrm{H}_{\mu^{\star}}(G) \cap \mathrm{L}_{x}(G,X^{n}) \subseteq \mathbb{R}$ by Lemma~\ref{lemma:isometric.action.harmonicity}. Furthermore, $\mathrm{L}_{x}(G,X^{n})$ is an introverted subset of $\mathrm{RUCB}(G)$ by Lemma~\ref{lemma:metric.stabilizer.introverted}. Hence, Lemma~\ref{lemma:boundary.triviality} asserts the existence of some $\nu \in \Sigma_{\mu^{\star}}(G)$ such that $\nu ({f\!\!\upharpoonright_{x}} \circ \lambda_{g}) = \nu ({f\!\!\upharpoonright_{x}})$ for all $f \in \mathrm{Lip}_{1}^{1}(X^{n})$ and $g \in G$. Thanks to Corollary~\ref{corollary:isometry.groups}, this shows that $G$ is amenable. \end{proof}

\section{The Liouville property for permutation groups}\label{section:thompson}

This section contains a brief discussion of consequences of our results for non-archimedean second-countable topological groups, i.e., those topologically isomorphic to a subgroup of the full symmetric group $\mathrm{Sym}(X)$, over a countable set $X$, equipped with the topology of pointwise convergence. The subsequent result follows immediately from Theorem~\ref{theorem:juschenko}.

\begin{cor}\label{corollary:juschenko.1} Let $X$ be a countable set. A topological subgroup $G$ of $\mathrm{Sym}(X)$ is amenable if and only if, for all $n \in \mathbb{N}$ and $x \in X^{n}$, the action of $G$ on $Gx$ is Liouville. \end{cor}

Given a group $G$ acting on a set $X$, one may consider the induced action of $G$ on the corresponding powerset $\mathscr{P}(X)$, defined via $gB \, \defeq \, \{ gx \mid x \in B \}$ for all $g \in G$ and $B \subseteq X$. Evidently, for every $n \in \mathbb{N}$, the set $\mathscr{P}\!_{n}(X)$ of all $n$-element subsets of $X$ then constitutes a $G$-invariant subset of $\mathscr{P}(X)$. We say that $G$ acts \emph{strongly transitively} on $X$ if, for every~$n \in \mathbb{N}$, the induced action of $G$ upon $\mathscr{P}\!_{n}(X)$ is transitive. Specializing Corollary~\ref{corollary:juschenko.1} to the case of groups of automorphisms of linearly ordered sets, we arrive at our next result. 

\begin{cor}\label{corollary:juschenko.2} Let $X$ be a countable set and let $G$ be a topological subgroup of $\mathrm{Sym}(X)$. Then the following hold. \begin{enumerate}
	\item[$(1)$] If the topological group $G$ is amenable, then the action of $G$ on $\{ gB \mid g \in G \}$ is Liouville for every finite subset $B \subseteq X$.
	\item[$(2)$] If $G$ preserves a linear order on $X$ and the action of $G$ on $\{ gB \mid g \in G \}$ is Liouville for every finite subset $B \subseteq X$, then the topological group $G$ is amenable.
\end{enumerate} \end{cor}

\begin{proof} For every $n \in \mathbb{N}$ and $x \in X^{n}$, the map \begin{displaymath}
	\phi_{x} \colon \, Gx \, \longrightarrow \, \{ g\{ x_{1},\ldots,x_{n}\} \mid g \in G \} , \quad y \, \longmapsto \, \{ y_{1},\ldots,y_{n} \} 
\end{displaymath} constitutes a $G$-equivariant surjection. Hence, (1) is a consequence of Corollary~\ref{corollary:juschenko.1}. Moreover, if $G$ preserves a linear order on $X$, then $\phi_{x}$ is a bijection for every $x \in X$, and therefore (2) also follows by Corollary~\ref{corollary:juschenko.1}. \end{proof}

Due to the seminal work of Kechris, Pestov and Todorcevic~\cite{KPT},  Corollary~\ref{corollary:juschenko.2} identifies structural Ramsey theory as a source of Liouville actions. More precisely, if $G$ is the automorphism group of an order Fra\"iss\'e structure on a countable set $X$ having the Ramsey property, then Corollary~\ref{corollary:juschenko.2} combined with~\cite[Theorem~4.7]{KPT} (see also~\cite[Corollary~6.6.18]{PestovBook}) asserts that the action of $G$ on $\{ gB \mid g \in G \}$ is Liouville for every finite $B \subseteq X$. A concrete example of a corresponding application concerning Thompson's group $F$ is given by Corollary~\ref{corollary:juschenko.3} below. We will formulate Corollary~\ref{corollary:juschenko.3} as a consequence of the following abstract result, which follows immediately from Corollary~\ref{corollary:juschenko.2} and Pestov's work on the extreme amenability of automorpism groups ultrahomogeneous linear orders~\cite{pestov}.

\begin{cor}\label{corollary:pestov} Let $X$ be a countable set and let $G$ be a subgroup of $\mathrm{Sym}(X)$. If the action of $G$ on $X$ is strongly transitive and preserves a linear order on $X$, then, for every~$n \in \mathbb{N}$, the action of $G$ on $\mathscr{P}\!_{n}(X)$ is Liouville. \end{cor}

\begin{proof} Our hypotheses about the action imply that $X$ is either empty, a singleton set, or infinite. If $X = \emptyset$ or $\vert X \vert = 1$, then the desired conclusion is trivial. Now suppose that $X$ is infinite. Then Pestov's work~\cite[Theorem~5.4]{pestov} asserts that the topological group $G$, carrying the topology of pointwise convergence, is (extremely) amenable. As $G$ acts strongly transitively on $X$, an application Corollary~\ref{corollary:juschenko.2}(1) finishes the proof. \end{proof}

In order to explain the above-mentioned application of our results, let us finally turn to Richard Thompson's group \begin{displaymath}
	F \, \defeq \, \left\langle \sigma, \tau \left\vert \, \left[\sigma \tau^{-1}, \sigma^{-1}\tau \sigma \right] = \left[\sigma \tau^{-1}, \sigma^{-2}\tau \sigma^{2} \right] = e \right\rangle , \right.
\end{displaymath} which possesses an alternative presentation given by \begin{displaymath}
	F \, \cong \, \left\langle (\gamma_{n})_{n \in \mathbb{N}} \left\vert \, \forall m,n \in \mathbb{N}, \, m<n \colon \ \gamma_{m}^{-1}\gamma_{n}\gamma_{m} \, = \, \gamma_{n} \right\rangle \right.
\end{displaymath} and corresponding to the previous one via $\gamma_{0} = \sigma$ and $\gamma_{n} = \sigma^{1-n}\tau \sigma^{n-1}$ for every $n \in \mathbb{N}_{\geq 1}$. For general background on this group, the reader is referred to~\cite{CannonFloydParry}. In the following, we will be particularly concerned with two of its representations. First of all, let us recall that Thompson's group $F$ admits a natural embedding into the group $\mathrm{Homeo}_{+}[0,1]$ of orientation-preserving homeomorphisms of the real interval $[0,1]$ determined by \begin{align*}
	& \sigma (x) \, \defeq \, \begin{cases}
		\, \tfrac{x}{2} & \left( x \in \left[ 0, \tfrac{1}{2} \right] \right) , \smallskip \\
		\, x - \tfrac{1}{4} & \left( x \in \left[ \tfrac{1}{2} , \tfrac{3}{4} \right] \right) , \smallskip \\
		\, 2x -1 & \left( x \in \left[ \tfrac{3}{4} , 1 \right] \right) ,
	\end{cases} & \tau (x) \, \defeq \, \begin{cases}
		\, x & \left( x \in \left[ 0, \tfrac{1}{2} \right] \right) , \smallskip \\
		\, \tfrac{x}{2} + \tfrac{1}{4} & \left( x \in \left[ \tfrac{1}{2} , \tfrac{3}{4} \right] \right) , \smallskip \\
		\, x - \tfrac{1}{8} & \left( x \in \left[ \tfrac{3}{4} , \tfrac{7}{8} \right] \right) , \smallskip \\
		\, 2x -1 & \left( x \in \left[ \tfrac{7}{8} , 1 \right] \right) .
	\end{cases}
\end{align*} The image of $F$ under this embedding consists of those elements of $\mathrm{Homeo}_{+}[0,1]$ which are piecewise affine (with finitely many pieces) and have all their break points contained in $\mathbb{Z}\!\left[ \tfrac{1}{2} \right]$ and slopes contained in the set $\! \left. \left\{ 2^{k} \, \right\vert k \in \mathbb{Z} \right\}$. Furthermore, by work of Brin and Squier~\cite{BrinSquier}, Thompson's group $F$ embeds into the group $\mathrm{Homeo}_{+}(\mathbb{R})$ of orientation-preserving homeomorphisms of the real line via the action on $\mathbb{R}$ given by \begin{align*}
	& \sigma (x) \, \defeq \, x-1 \quad \, (x \in \mathbb{R}), & \tau (x) \, \defeq \, \begin{cases}
		\, x & (x \in (-\infty, x]) , \smallskip \\
		\, \tfrac{x}{2} & (x \in [0,2]) , \smallskip \\
		\, x - 1 & (x \in [2, \infty )).
	\end{cases}
\end{align*} The image of this second embedding consists of those members of $\mathrm{Homeo}_{+}(\mathbb{R})$ which, again, are piecewise affine (with finitely many pieces) and have all their break points inside $\mathbb{Z}\!\left[ \tfrac{1}{2} \right]$ and slopes inside $\! \left. \left\{ 2^{k} \, \right\vert k \in \mathbb{Z} \right\}$, and which moreover agree with translations by (possibly two different) integers on $(-\infty,a]$ and $[a,\infty)$ for a sufficiently large dyadic rational $a \in \mathbb{Z}\!\left[ \tfrac{1}{2} \right]$.

The two embeddings of Thompson's group $F$ introduced above are connected in a fairly natural way. As remarked in~\cite[Remark~2.5]{haagerup} (see also~\cite[2.C]{kaimanovich}), the piecewise affine map $\kappa \colon (0,1) \to \mathbb{R}$ given by \begin{displaymath}
	\kappa (x) \, \defeq \, \tfrac{x-t_{n}}{t_{n+1}-t_{n}} + n \qquad (x \in [t_{n},t_{n+1}], \, n \in \mathbb{Z}) ,
\end{displaymath} where $t_{n} \defeq 1 - \tfrac{1}{2^{n+1}}$ for $n \in \mathbb{N}$ and $t_{n} \defeq \tfrac{1}{2^{1-n}}$ for $n \in \mathbb{Z}\setminus \mathbb{N}$, constitutes an $F$-equivariant monotone bijection between $(0,1)$ and $\mathbb{R}$. Furthermore, $\kappa (D) = \mathbb{Z}\! \left[ \tfrac{1}{2} \right]$ for $D \defeq (0,1) \cap \mathbb{Z}\!\left[\tfrac{1}{2}\right]$, whence the restricted $F$-actions on $\mathbb{Z}\!\left[ \tfrac{1}{2} \right]$ and $D$ are isomorphic.

Since the considered actions of Thompson's group $F$ on $\mathbb{Z}\!\left[ \tfrac{1}{2} \right]$ and $D$ are strongly transitive and preserve the natural linear order, our Corollary~\ref{corollary:pestov} entails the following affirmative answer to a recent question by Juschenko~\cite{juschenko}, motivated by work of Kaimanovich~\cite{kaimanovich}.

\begin{cor}\label{corollary:juschenko.3} For all $n \in \mathbb{N}$, the action of $F$ on $\mathscr{P}\!_{n}\!\left(\mathbb{Z}\!\left[\tfrac{1}{2}\right]\right)$ (resp., $\mathscr{P}\!_{n}(D)$) is Liouville. \end{cor}

Furthermore, Corollary~\ref{corollary:juschenko.2} resolves another recent problem by Juschenko~\cite{juschenko}. 

\begin{problem}[\cite{juschenko}, Problem~9]\label{problem:juschenko} Let a group $G$ act faithfully on a countable set $X$ such that, for every $n \in \mathbb{N}$, the induced action of $G$ on $\mathscr{P}\!_{n}(X)$ is Liouville. Is $G$ amenable? \end{problem}

Our results entail that the solution to Problem~\ref{problem:juschenko} is negative, even if we require $G$ to be countable. Indeed, if $H$ is any amenable topological subgroup of $\mathrm{Sym}(X)$ acting strongly transitively on $X$ and containing a dense countable subgroup $G$ which is non-amenable as a discrete group (e.g.,~isomorphic to the free group $F_{2}$ on two generators), then, however, the density will imply that the action of $G$ on $X$ is strongly transitive and that the topological subgroup $G \leq H$ is amenable, whence by Corollary~\ref{corollary:juschenko.2}(1) the action of $G$ on $\mathscr{P}\!_{n}(X)$ will be Liouville for every~$n \in \mathbb{N}$. For instance, the topological group $\mathrm{Sym}(\mathbb{N})$ is amenable, acts strongly transitively on $\mathbb{N}$, and contains dense subgroups isomorphic to $F_{2}$~\cite{McDonough} (see also~\cite{dixon}). For another example, we note that the topological group $\mathrm{Aut}(\mathbb{Q},{\leq})$ is (even extremely) amenable~\cite{pestov}, acts strongly transitively on $\mathbb{Q}$~\cite[p.~140]{cameron}, and contains a dense subgroup isomorphic to $F_{2}$~\cite{GGS}.

On the other hand, we note that if a group $G$ acts faithfully on a countable set $X$ in such a way that the action preserves a linear order on $X$ and, for every $n \in \mathbb{N}$, the induced action of $G$ on $\mathscr{P}\!_{n}(X)$ is Liouville, then the \emph{topological group} $G$, carrying the subspace topology inherited from $\mathrm{Sym}(X)$, is indeed amenable. This is a consequence of our Corollary~\ref{corollary:juschenko.2}(2) (combined with Remark~\ref{remark:liouville}(2)).

\section*{Acknowledgments}

The first-named author acknowledges funding of the Excellence Initiative by the German Federal and State Governments as well as the Brazilian Conselho Nacional de Desenvolvimento Cient\'{i}fico e Tecnol\'{o}gico (CNPq), processo 150929/2017-0. The second-named author acknowledges funding by the ERC Consolidator Grant No.~681207. The authors would like to express their gratitude towards Kate Juschenko, Nicolas Monod, Jan Pachl, Vladimir Pestov, Andy Zucker, as well as the anonymous referee for a number of insightful comments on earlier versions of this paper. A remark by Andy Zucker has led to the generalization of an earlier version of Proposition~\ref{proposition:poisson.algebra}. Moreover, the authors are deeply indebted to Jan Pachl for pointing out an error in a previous version of Lemma~\ref{lemma:boundary.triviality}.


\enlargethispage{0.4cm}

\end{document}